\documentclass[11pt,reqno]{amsart}

\usepackage{amsmath}
\usepackage{amsthm}
\usepackage{amsfonts}
\usepackage{amssymb}
\usepackage{color}
\usepackage{graphicx}
\usepackage[hidelinks]{hyperref}
\usepackage[margin=1.0in]{geometry}
\usepackage{todonotes}
\usepackage{verbatim}
\usepackage{comment}
\numberwithin{equation}{section}
\usepackage{enumitem}
\usepackage{caption}
\usepackage{subcaption}
\usepackage{epstopdf}

\definecolor{skyblue}{rgb}{0.85,0.85,1}

\newtheorem{lemma}{Lemma}

\newtheorem{theorem}{Theorem}
\newtheorem{cor}{Corollary}

\newtheorem{rem}{Remark}

\newtheorem{define}{Definition}
\numberwithin{lemma}{section}
\numberwithin{prop}{subsection}
\numberwithin{theorem}{section}
\numberwithin{cor}{section}
\numberwithin{conj}{section}
\numberwithin{rem}{section}

\DeclareMathOperator{\Mor}{Mor}

\newcommand{\bbC}{\mathbb{C}}

\newcommand{\bbR}{\mathbb{R}}

\newcommand{\bbZ}{\mathbb{Z}}
\newcommand{\vp}{\varphi}

\newcommand{\p}{\partial}

\DeclareMathOperator{\Gr}{Gr}

\newcommand{\eps}{\epsilon}
\newcommand{\Maslov}{\mathrm{Maslov}(\vp)}

\begin{document}

\title[Existence and stability of FitzHugh-Nagumo traveling waves]{On the Existence and Stability of Fast Traveling Waves in a Doubly-Diffusive FitzHugh-Nagumo System}
\author[P. Cornwell]{Paul Cornwell}
\email{pcorn@live.unc.edu}
\address{Department of Mathematics, UNC Chapel Hill, Phillips Hall CB \#3250, Chapel Hill, NC 27516}
\author[C.\,K.\,R.\,T.\ Jones]{Christopher K.\,R.\,T.\ Jones}
\email{ckrtj@email.unc.edu}

\begin{abstract}
The FitzHugh-Nagumo equation, which was derived as a simplification of the Hodgkin-Huxley model for nerve impulse propagation, has been extensively studied as a paradigmatic activator-inhibitor system. We consider the version of this system in which two agents diffuse at an equal rate. Using geometric singular perturbation theory, we prove the existence and stability of fast traveling pulses. The stability proof makes use of the Maslov index--an invariant of symplectic geometry--to count unstable eigenvalues for the linearization about the wave. The calculation of the Maslov index is carried out by tracking the evolution of the unstable manifold of the rest state using the timescale separation. This entails a careful consideration of how the transition from fast to slow dynamics occurs in the tangent bundle over the wave. Finally, we observe in the calculation that the Maslov index lacks monotonicity in the spatial parameter, which distinguishes this application of the Maslov index from similar analyses of Hamiltonian systems.
\end{abstract}

\maketitle

\tableofcontents

\section{Introduction}
Activator-inhibitor systems are of great interest to the scientific community as breeding grounds for patterns, traveling waves, and other localized structures. As shown by Turing \cite{Turing}, such structures often arise from the destabilization of a stable equilibrium in the presence of diffusion. A crucial mathematical question is whether these structures are stable as solutions of the PDE, since this intuitively determines whether they can be observed in nature. In this work, we investigate the stability of traveling pulses for a FitzHugh-Nagumo system. The main tool is an invariant of symplectic geometry called the Maslov index, which is used to prove the nonexistence of unstable eigenvalues for the linearization about the wave. The Maslov index is encoded in the twisting of the unstable manifold of the rest state, which we are able to calculate using techniques of geometric singular perturbation theory. The framework for the use of the Maslov index was established recently \cite{corn} for activator-inhibitor systems. One of the interesting contrasts between activator-inhibitor and gradient reaction-diffusion equations is that the Maslov index lacks monotonicity in the spatial parameter. As such, we observe multiple offsetting conjugate points in the Maslov index calculation.

The setting is the system of equations \begin{equation}\label{general PDE}
\begin{aligned} u_t & =u_{xx}+f(u)-v\\
v_t & = v_{xx}+\eps(u-\gamma v),
\end{aligned}
\end{equation} where $u,v\in\bbR$, and $x,t\in\bbR$ are space and time respectively. The function $f$ is the ``bistable'' nonlinearity $f(u)=u(1-u)(u-a)$, where $0<a<1/2$ is constant. We take $\eps>0$ to be very small, making this a singular perturbation problem. The stability of various traveling and standing fronts and pulses has been studied for variations of (\ref{general PDE}) in which there is either no diffusion on $v$, or the diffusion coefficient is a small parameter \cite{Jones84,Flo91,CdRS16,Yan85,AGJ,CH14}. An existence result for traveling pulses of (\ref{general PDE}) was obtained in \cite{CC15} using variational methods, but the issue of stability was unresolved. We offer another existence proof using geometric singular perturbation theory. The value of this proof is that the construction of the pulse provides the means for assessing the stability of the wave using the Maslov index.

Setting $z=x-ct$ to obtain a moving frame, a traveling wave $\varphi(z)=(\hat{u}(z),\hat{v}(z))$ is a steady state (i.e. $\hat{u}_t=\hat{v}_t\equiv0$) of the equation
\begin{equation}\label{traveling wave PDE}
\begin{aligned} u_t & =u_{zz}+cu_z+f(u)-v\\
v_t & = v_{zz}+cv_z+\eps(u-\gamma v),
\end{aligned}
\end{equation} which decays exponentially as $z\rightarrow\pm\infty$. We can then introduce the variables $u_z=w$ and $v_z=\eps y$ to convert the steady state equation for (\ref{traveling wave PDE}) to the first-order system \begin{equation}\label{FHN traveling wave ODE}
U_z=\left(\begin{array}{c}
u\\
v\\
w\\
y
\end{array}\right)_z=\left(\begin{array}{c}
w\\
\eps y\\
-cw-f(u)+v\\
-cy+\gamma v-u
\end{array}\right)=F(U).
\end{equation} This transformation makes the traveling wave a homoclinic orbit to a fixed point for (\ref{FHN traveling wave ODE}). We will use $\vp$ as a label for both the solution $\vp=(\hat{u},\hat{v})$ to (\ref{traveling wave PDE}) and the solution $\vp=(\hat{u},\hat{v},\hat{u}',\hat{v}'/\eps)$ of (\ref{FHN traveling wave ODE}); it will be clear from context which object is being referenced. We assume that the parameter $\gamma>0$ is small enough so that (\ref{FHN traveling wave ODE}) has only one fixed point at the origin. It follows that $\vp$--if it exists--must be homoclinic to $0\in\bbR^4$. The construction proves the stronger statement that $\vp$ is \emph{transversely constructed} in the following sense. With the equation $c'=0$ appended to (\ref{FHN traveling wave ODE}), the center-stable and center-unstable manifolds of $0$ are each three-dimensional. They intersect transversely in a one-dimensional set in $\bbR^5$, and $\vp$ is precisely this intersection. The Exchange Lemma \cite{JK94,Jones_GSP,JKT94} is key in proving that this transverse intersection exists.

Since (\ref{FHN traveling wave ODE}) is autonomous, any translate $\vp(z+k)$, $k\in\bbR$, of $\vp$ is also a traveling pulse solution of (\ref{traveling wave PDE}). Taking this into account, we use the following definition for the stability of $\vp$. \begin{define}\label{stability defn}
	The traveling wave $\vp(z)$ is \textbf{asymptotically stable} relative to (\ref{traveling wave PDE}) if there is a neighborhood $V\subset BU(\bbR,\bbR^2)$ of $\vp(z)$ such that if $u(z,t)$ solves (\ref{traveling wave PDE}) with $u(z,0)\in V$, then \[||\vp(z+k)-u(z,t)||_\infty\rightarrow 0 \] as $t\rightarrow\infty$ for some $k\in\bbR$.
\end{define} It is well-known \cite{BJ89,Henry} that the stability of such a wave is determined by the spectrum of operator \begin{equation}\label{linearized operator}
L=\p_z^2+c\p_z+\left(\begin{array}{c c}
f'(\hat{u}) & -1\\
\eps & -\eps\gamma
\end{array}\right)
\end{equation} obtained by linearizing (\ref{traveling wave PDE}) around $\varphi$. It is shown in \cite{CJ17} that the essential spectrum of $L$ is bounded away from the imaginary axis in the left half-plane, and thus the stability of $\varphi$ is determined by the eigenvalues of $L$. In \cite{CJ17}, it was established that there is a symplectic structure to the eigenvalue equations that allows one to define a Maslov index for $\vp$. It was then shown in \cite{corn} that the Maslov index gives an exact count of all unstable eigenvalues for $L$. Proving that $\vp$ is stable is therefore tantamount to proving that $\Maslov=0$.

The Maslov index \cite{Maslov,Arnold67,Arnold85} is a homotopy invariant assigned to curves of Lagrangian subspaces. Roughly speaking, the index counts how many times a curve of subspaces intersects the \emph{train} of a fixed subspace. Its relevance to stability analysis stems from the fact that the eigenvalue problem $Lp=\lambda p$ is a linear ODE, which induces equations on Grassmannians of all dimensions. For $\lambda\in\bbR$, the set of Lagrangian planes, called the \emph{Lagrangian Grassmannian} $\Lambda(2)$, is shown to be an invariant manifold for the equation induced on $\Gr_2(\bbR^4)$. Among the elements in $\Lambda(2)$ are the stable and unstable bundles, $E^s(\lambda,z)$ and $E^u(\lambda,z)$, consisting of the solutions that decay as $z\rightarrow\infty$ and as $z\rightarrow-\infty$ respectively. The limits of these solution spaces as elements in $\Lambda(2)$ are known (see \S 3 of \cite{AGJ}), so the eigenvalue problem can be recast as the problem of finding connecting curve segments between points in this space. This is the perspective of \cite{corn}, in which it is shown using a homotopy argument that the Maslov index of the curve $z\mapsto E^u(0,z)$ counts the number of unstable eigenvalues of $L$.

The intuition behind the application of the Maslov index described above comes from Sturm-Liouville theory. Consider the scalar equation $u_t=u_{xx}+F(u)$. Suppose that there is a steady front or pulse solution $\hat{u}$ to this equation, which corresponds to a heteroclinic or homoclinic orbit connecting two equilibria in phase space. It is a classic result (for example, \S 2.3 of \cite{KP13}) that the number of unstable eigenvalues associated with the linearization about this wave is equal to the number of critical points of the solution itself. Although there are several ways to prove this, the key to each proof is the fact that the derivative of the wave is a $0$-eigenfunction for the corresponding linear operator. Its zeros (i.e. critical points of $\hat{u}$) can be thought of as intersections of the unstable bundle (one-dimensional in this case) with the vertical subspace $\{0\}\times\bbR\subset\bbR^2$. 

In moving from a scalar equation to (\ref{general PDE}), the unstable bundle becomes two-dimensional, and the derivative of the wave is just one vector in that space. Nonetheless, a connection between the eigenvalue equation and the traveling wave ODE (\ref{FHN traveling wave ODE}) persists. Written as a first order system, the eigenvalue equation $Lp=0$ becomes \begin{equation}\label{linearization around wave}
\left(\begin{array}{c}
p\\
q\\
r\\
s
\end{array}\right)'=\left(\begin{array}{c c c c }
0 & 0 & 1 & 0\\
0 & 0 & 0 & \eps\\
-f'(\hat{u}) & 1 & -c & 0\\
-1 & \gamma & 0 & -c
\end{array}\right)\left(\begin{array}{c}
p\\q\\r\\s
\end{array}\right).
\end{equation} This is precisely the variational equation for (\ref{FHN traveling wave ODE}) along $\vp$. It is then known (cf. \S 6 of \cite{CJ17}) that the unstable bundle $E^u(0,z)$ is parallel to the tangent space $T_{\vp(z)}W^u(0)$ for all $z\in\bbR$. On the one hand, this is the motivation for studying the Maslov index in this context; it is possible to ascertain spectral information from how the wave (or, in this case, an invariant manifold associated with the wave) is situated in phase space. On the other hand, the relationship between the linear and nonlinear equations provides a means to the end of \emph{calculating} the Maslov index if, for some reason, more information about the nonlinear object $W^u(0)$ is known than about the linear object $E^u(0,z)$. For example, in singularly perturbed systems, Fenichel theory \cite{Fen79,Jones_GSP} and subsequent developments provide a means for tracking invariant manifolds throughout phase space. By contrast, the timescale separation makes the eigenvalue problem itself no more or less tractable.

The thrust of this paper is therefore to calculate the Maslov index by using the timescale separation of (\ref{FHN traveling wave ODE}). To our knowledge, this is the first instance in which a complete calculation of the Maslov index is used to prove that a traveling wave for a system of equations is stable. In \cite{BJ}, the Maslov index is used to study the stability of traveling waves for a FitzHugh-Nagumo system coupled to an ancillary ODE. However, in that case the Evans function--not the Maslov index--is the primary ingredient in the proof. Arguments similar to \cite{Jones84} are used to show that the only potentially unstable eigenvalues are close to eigenvalues of traveling fronts for two reduced systems. The Maslov index is then used to study these reduced systems. In this work, the Maslov index is defined and calculated for the full traveling wave, and it alone is used to obtain the stability result. The Maslov index is also used in \cite{BCJLMS17} to study standing waves in gradient reaction-diffusion equations. In that work, the reversibility of the standing wave equation is used to prove the existence of a conjugate point, which in turn proves the existence of an unstable eigenvalue. However, this provides a lower bound on the number of unstable eigenvalues, as opposed to the exact count that we obtain. Moreover, we study traveling waves, and the reversibility symmetry is not present in this context.

A recent work \cite{CH14} studied standing waves for (\ref{general PDE}) using the Maslov index. Such waves are obtained as local minimizers of an energy functional, and the Maslov index is used in conjunction with this variational formulation to aid in the calculation of spectral flow for a family of self-adjoint operators. This is very different from the strategy employed in this paper, and indeed it alone is insufficient for determining the stability of the traveling waves found in \cite{CC15}. However, many of the ideas and techniques of \cite{CH14} are crucial to establishing the framework of \cite{corn} for using the Maslov index to count eigenvalues of $L$. Most notably, the proof that any unstable spectrum must be real can be adapted from the case of standing waves to that of traveling waves.

As discussed in \cite{corn}, system (\ref{traveling wave PDE}) is a \emph{skew-gradient} system, see \cite{Yan02a,Yan02b}. One of the challenges that distinguishes skew-gradient from gradient reaction-diffusion equations is that the Maslov index is generally not monotone in its parameters. This means that the curve in question can cross the train of the reference plane in different directions. Estimates relevant to the variational existence proof (cf. \S 2 of \cite{CC15}) can be used to show that the Maslov index for (\ref{general PDE}) is monotone in the spectral parameter $\lambda$. This is crucial in the ``Maslov = Morse" theorem of \cite{corn} which is used to prove stability in this work. However, we will see in \S 4 that monotonicity in $z$ is absent. Indeed, we will show that there are four conjugate points for the unstable bundle, but they offset in the Maslov index calculation to give $\Maslov=0$. This is partly what makes this example interesting, since it removes the possibility of using arguments as in \cite{BCJLMS17} to conclude that the Maslov index is nonzero from the existence of a single conjugate point. Nonetheless, \emph{as long as one can find all conjugate points}, monotonicity in the spatial parameter is not needed to make use of the Maslov index.

The rest of this paper is organized as follows. In \S 2, we prove the existence of the wave using geometric singular perturbation theory and the Exchange Lemma. That section contains a careful description of the ``singular orbit," which is useful in the Maslov index calculation. As a corollary, we also prove the existence of front solutions of (\ref{general PDE}) in different parameter regimes. In \S 3, the Maslov index of the traveling wave is defined, and the relevant theorems on stability are given. Finally, in \S 4 we complete the stability proof by calculating the Maslov index. Two appendices give background information on Pl\"{u}cker coordinates and the geometry of induced flows on Grassmannians.

\section{Existence of the Wave}
The existence of traveling pulse solutions of (\ref{general PDE}) was proved by Chen and Choi \cite{CC15} using variational techniques. We furnish another proof here using geometric singular perturbation theory, since the construction of the pulse in this manner is important for the stability analysis to follow. First, observe that the linearization about the fixed point $0$ is obtained by setting $u=0$ in (\ref{linearization around wave}). Since $f'(0)=-a$, a computation gives that the eigenvalues of the linearization at $0$ are given by \begin{equation}\label{evals of lin}\begin{aligned} \mu_1 & = -\frac{c}{2}-\frac{1}{2}\sqrt{c^2+2(\gamma\eps+a)+2\sqrt{(\gamma\eps-a)^2-4\eps}}\\
\mu_2 & = -\frac{c}{2}-\frac{1}{2}\sqrt{c^2+2(\gamma\eps+a)-2\sqrt{(\gamma\eps-a)^2-4\eps}}\\
\mu_3 & = -\frac{c}{2}+\frac{1}{2}\sqrt{c^2+2(\gamma\eps+a)-2\sqrt{(\gamma\eps-a)^2-4\eps}}\\
\mu_4 & = -\frac{c}{2}+\frac{1}{2}\sqrt{c^2+2(\gamma\eps+a)+2\sqrt{(\gamma\eps-a)^2-4\eps}}
\end{aligned}.
\end{equation} Supposing for the moment that $c<0$ is $O(1)$ in $\eps$ (which will be shown shortly), it is clear that for $\eps>0$ sufficiently small, we have \begin{equation}\label{eval inequality}
\mu_1<\mu_2<0<-c<\mu_3<\mu_4.
\end{equation} Furthermore, the eigenvalue $\mu_2$ approaches $0$ as $\eps\rightarrow0$. It follows that $W^u(0)$ and $W^s(0)$, the unstable and stable manifolds of $0$ respectively, are each two-dimensional. We denote by $V^s(0)$ and $V^u(0)$ the corresponding stable and unstable subspaces, which are tangent to $W^s(0)$ and $W^u(0)$ respectively at $0$.

The goal is to construct $\vp$ by showing that these manifolds intersect. Ideally, this would be accomplished by showing that the intersection exists when $\eps=0$, and then perturbing to the case $\eps>0$. However, we would need $W^u(0)$ and $W^s(0)$ to intersect transversely when $\eps=0$ to make this argument, as this would ensure that the intersection is not broken when $\eps$ is ``turned on." This is inevitably \emph{not} the case, since two two-dimensional submanifolds of $\bbR^4$ cannot intersect transversely in a one-dimensional set. To remedy this, we append the equation $c'=0$ to (\ref{FHN traveling wave ODE}) to obtain three-dimensional center-stable and center-unstable manifolds, $W^{cs}(0)$ and $W^{cu}(0)$. The phase space is now $\bbR^5$, and it follows from page 144 of \cite{Lee12} that the transverse intersection $W^{cu}(0)\pitchfork W^{cs}(0)$ is one-dimensional. Thus if we can prove that this transverse intersection exists when $\eps=0$, it would follow that it persists to the case $\eps>0$, proving the existence of the wave. Ironically, it will be necessary to use information from the perturbed system to conclude that the transverse intersection exists when $\eps=0$. The technical tool that makes this connection is the Exchange Lemma, which will be discussed in \S 2.3.

\subsection{Fast-Slow Structure}
By taking $0<\eps\ll1$, (\ref{FHN traveling wave ODE}) is seen to be singularly perturbed. Such systems are amenable to analysis by geometric dynamical systems techniques, owing to the work of Fenichel \cite{Fen79}. Broadly speaking, Fenichel theory provides a means for reconciling two reduced systems obtained by taking the limit $\eps\rightarrow0$ on different timescales (fast and slow). Introductions to this theory are found in \cite{Jones_GSP,Kuehn15}, and we refer the reader to these sources for explanations of the terminology used freely throughout this paper. 

The chosen scaling of $v_z$ makes (\ref{FHN traveling wave ODE}) a fast-slow system with three fast variables $(u,w,y)$ and one slow variable $v$. Setting $\eps=0$ in (\ref{FHN traveling wave ODE}), one arrives at the so-called layer problem \begin{equation}\label{FHN layer problem}
 \left(\begin{array}{c}
 u\\
 w\\
 y
 \end{array} \right)'=\left(\begin{array}{c}
 w\\
 -cw+v-f(u)\\
 -cy+\gamma v-u
 \end{array} \right).
 \end{equation} This is a smooth limit, and now $v$ plays the role of a parameter. Trajectories for (\ref{FHN layer problem}) will be close to those of (\ref{FHN traveling wave ODE}) (with the $v$-component included), provided they are not near critical points. The problem is that taking this limit actually generates a one-dimensional set of critical points for (\ref{FHN layer problem}). These points comprise the critical manifold $M_0$, which is given by \begin{equation}\label{critical mfd}
M_0=\{(u,v,w,y):v=f(u),w=0,y=\frac{1}{c}(\gamma v-u)\}.
\end{equation} $M_0$ is normally hyperbolic wherever $f'(u)\neq 0$, meaning that the linearization of (\ref{FHN layer problem}) about any point in $M_0$ has no eigenvalues with $0$ real part. The layer equation does not tell us anything about motion on $M_0$, since each point therein is fixed by definition. Instead, setting $\zeta=\eps z$ in (\ref{FHN traveling wave ODE}), the limit $\eps\rightarrow 0$ yields the slow flow \begin{equation}\label{slow flow}
\dot{v}=y=\frac{1}{c}(\gamma v-f^{-1}(v)),\hspace{.2 in}\left(\dot{ }=\frac{d}{d\zeta}\right)
\end{equation} which is restricted to $M_0$. By $f^{-1}$, we mean the inverse of $f$ restricted to one of three segments of the cubic $v=f(u)$, partitioned by the two zeros of $f'(u)$. Of particular interest are the two outer branches corresponding to the intervals on which $f(u)$ is strictly decreasing. We use the notation $M_0^L$ and $M_0^R$ for the left and right branches respectively.

Fenichel's Theorem asserts that the manifold $M_0$ perturbs to a nearby version $M_\eps$, provided that we are away from zeros of $f'$ (i.e. where normal hyperbolicity fails). This manifold is locally invariant and the flow on $M_\eps$ is given to leading order by (\ref{slow flow}). Actually, Fenichel derived a much stronger result. By computing the eigenvalues of the matrix in (\ref{linearization around wave}) as a function of $u$, one sees that each fixed point in $M_0^{L/R}$ for (\ref{FHN layer problem}) has two unstable and one stable eigenvalues, each of which generates an invariant manifold. Taking the union over $v$ in some compact subset of $M_0^{L/R}$, one obtains a three-dimensional $W^{u}(M_0^{L/R})$ and two-dimensional $W^{s}(M_0^{L/R})$ as subsets of $\bbR^4$. (By abuse of notation, $M_0^{L/R}$ here refers to those compact subsets.) The second result of Fenichel is that these invariant manifolds also perturb to locally invariant objects $W^{u/s}(M_\eps^{L/R})$. These objects play a crucial role in both the existence of the pulse and the calculation of the Maslov index.

\subsection{The Singular Solution}
Both the existence and stability of the fast traveling waves are heavily informed by the structure of the singular orbit which forms their template. It is therefore instructive to devote time to this object. The singular pulse is very similar to the corresponding object with no diffusion on $v$, see \cite{Jones84,JKL91}. We call that system the `3D system,' in reference to the dimension of phase space of the traveling wave equation. We begin with the layer problem (\ref{FHN layer problem}).
 Notice that the equations for $u$ and $w$ decouple from $y$, so the projection of any solution of (\ref{FHN layer problem}) onto the $uw-$plane will be (part of) a solution to \begin{equation}\label{3d layer problem}
 \left(\begin{array}{c}
u\\ w \end{array}\right)'=\left(\begin{array}{c}
w\\-cw+v-f(u)
\end{array}\right).
 \end{equation} This system is considered in the construction of traveling waves for the 3D system. It is shown in \cite{McKean70} that for $v=0$ and  \begin{equation}\label{singular c value}
 c=c^*:=\sqrt{2}(a-1/2)<0,
 \end{equation} there exists a heteroclinic orbit connecting the fixed point $(0,0)$ with the fixed point $(1,0)$. The explicit solution is given by \begin{equation}\label{McKean soln}
 u(z)=\frac{1}{\left(1+e^{-\frac{\sqrt{2}}{2}z}\right)},
 \end{equation} which in turn determines $w$. One can solve for $w$ as a function of $u$ to get the profile in phase space. It is easy to verify that \begin{equation}\label{fast jump profile}
 w(u)=\frac{\sqrt{2}}{2}u(1-u), \hspace{.2 in}0\leq u\leq 1
 \end{equation} is the profile of the fast jump, together with the fixed points. To show that the corresponding connection exists for (\ref{FHN layer problem}), consider the linearization of (\ref{FHN layer problem}) about $0$: \begin{equation}\label{layer lin 0}
 \left(\begin{array}{c}
   \delta u\\ \delta w\\ \delta y
 \end{array}\right)'=\left(\begin{array}{c c c}
0 & 1 & 0\\
a & -c^* & 0\\
-1 & 0 & -c^*
 \end{array}\right)\left(\begin{array}{c}
\delta u\\ \delta w\\ \delta y
 \end{array}\right).
 \end{equation} The eigenvalues of the matrix in (\ref{layer lin 0}) are \begin{equation}\label{layer lin 0 evals}
 \mu_1(0)=-a\sqrt{2}, \, \mu_3(0)=-c^*,\text{ and }\mu_4(0)=\frac{\sqrt{2}}{2},
 \end{equation} which one would also obtain by substituting $\eps=0$ and $c=\sqrt{2}(a-1/2)$ in (\ref{evals of lin}). We therefore have a two-dimensional unstable manifold, and we wish to find an intersection with the stable manifold of $p=(1,0,-1/c^*)$. As noted above, the $y$ direction is invariant--one can see from (\ref{layer lin 0}) that $[0,0,1]^T$ is the $-c-$eigenvector--so the unstable manifold is actually a cylinder over the heteroclinic connection for (\ref{3d layer problem}). To avoid confusion with the unstable manifold for (\ref{FHN traveling wave ODE}), we will call this set $W^u(0_f)$, where the subscript indicates that this is the origin for the fast subsystem (\ref{FHN layer problem}).
  
 If we now linearize about the landing point $p$, we obtain \begin{equation}\label{layer lin 1}
  \left(\begin{array}{c}
    \delta u\\ \delta w\\ \delta y
  \end{array}\right)'=\left(\begin{array}{c c c}
 0 & 1 & 0\\
 1-a & -c^* & 0\\
 -1 & 0 & -c^*
  \end{array}\right)\left(\begin{array}{c}
 \delta u\\ \delta w\\ \delta y
  \end{array}\right),
  \end{equation}
which still has two unstable and one stable eigenvalues given by \begin{equation}\label{layer lin 1 evals}
 \mu_1(p)=-\frac{\sqrt{2}}{2}, \, \mu_3(p)=-c,\text{ and }\mu_4(p)=\sqrt{2}(1-a).
 \end{equation} In particular, the $\delta y$ direction is still invariant and unstable. Now, any trajectory in the cylinder $W^u(0_f)$ (for $0<u<1$) must approach the invariant line $\{(u,w,y):u=1,w=0\}$ in forward time. Since this line moves points away from the equilibrium $p$, there will be some points in $W^u(0_f)$ for which $y\rightarrow+\infty$ and others for which $y\rightarrow-\infty$. A shooting argument in the cylinder therefore produces an orbit that is bounded, which can only approach $p$. This orbit coincides with the one-dimensional stable manifold of $p$.
 
 \begin{center}
 	\begin{figure}[h]
 		\includegraphics[scale=1]{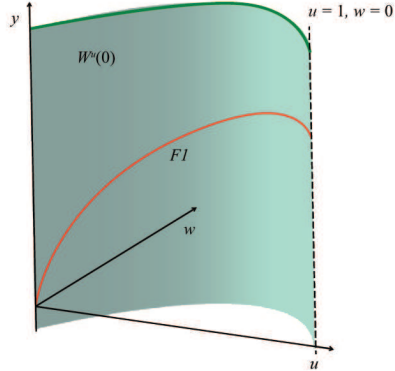}
 		
 		\caption{Unstable manifold $W^u(0)$ near the fast front with $v$ suppressed. The red curve is the fast jump.}
 		
 		\label{figfrontcyl}
 	\end{figure}
 \end{center} 

Before moving to the slow flow, we point out one important fact about the fast jump. Since $u(z)$ and $w(z)$ are also solutions of the 3D system, they must decay to $0$ as $z\rightarrow-\infty$ like $e^{\mu_4(0)z}$, since $\mu_4(0)$ is the only unstable eigenvalue for the linearization of (\ref{3d layer problem}) at $0$. On the other hand, $y$ satisfies \begin{equation}
 y'=-c^*y-u,
 \end{equation} which can be solved explicitly as \begin{equation}
 y(z)=Ke^{-c^*z}-e^{-c^*z}\int\limits_{-\infty}^{z}e^{c^*s}u(s)\,ds.
 \end{equation} The constant \begin{equation}\label{y soln constant}
 K=\int\limits_{-\infty}^{\infty}e^{c^*z}u(z)\,dz=\frac{2\pi}{\sqrt{2}\sin\left(\pi(1-2a)\right)}
 \end{equation} is uniquely determined by the requirement that $y$ is bounded at $\pm\infty$. It is then clear that \begin{equation}
 \lim\limits_{z\rightarrow-\infty}e^{c^*z}y(z)=K\neq 0.
 \end{equation} In other words, $u$ and $w$ decay more quickly than $y$ in backwards time, so the traveling front is asymptotically tangent to the $y-$axis. As far as traveling waves go, this is to be expected, since the velocity is generically tangent to the leading unstable direction in reverse time, see \cite{HomSan10}.
 
 Having established the existence of a connection between the left and right branches of $M_0$, the next step is to follow the slow flow up $M_0^R$. We rescale the independent variable as $\zeta=\eps z$ and again set $\eps=0$, to arrive at the reduced problem \begin{equation}
 \dot{v}=y, \hspace{.1 in}\left(\dot{ }=\frac{d}{d\zeta}\right)
 \end{equation} which is restricted to the critical manifold $M_0$. Since $y=-1/c^*>0$ at the landing point, $v$ will increase and move up the graph of the cubic $f$. Eventually it will reach a point $v^*$, for which (\ref{3d layer problem}) with $v=v^*$ and $c=c^*$ has a heteroclinic connection back to $M_0^L$. Using symmetries of the cubic, it can be shown that \begin{equation}\label{u-star}
 v^*=f(2/3(a+1)):=f(u^*).
 \end{equation} (See \S 3.1 of \cite{CdRS16} for more details.) We call the point $q:=(u^*,v^*,0,1/c^*(\gamma v^*-u^*))$ the \emph{jump-off} point for $M_0^R$. The same shooting argument as for the front proves that the back exists for (\ref{FHN layer problem}) as well. Finally, upon landing on $M_0^L$ at the point $u=2/3(a-1/2)$, the slow flow carries us back down to $0$, which is a fixed point for both the slow and fast systems.

\begin{center}
	\begin{figure}[h]
		\includegraphics[scale=1.5]{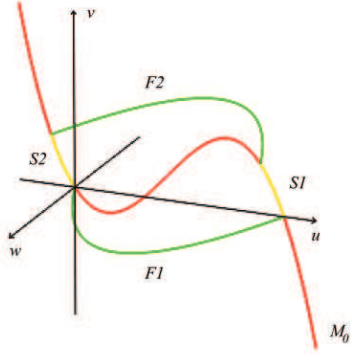}
		
		\caption{Picture of singular orbit, with $y$ suppressed. This orbit is identical to the singular orbit for the `3D system.' See \cite{JKL91}.}
		
		\label{fig3d system}
	\end{figure}
\end{center} 
\subsection{Transversality along the Front}

The pulse will ultimately be constructed by appealing to the Exchange Lemma \cite{JK94,KapJ01,JKT94}, which describes the passage of a shooting manifold near the right slow manifold (in this case $W^{cu}(0)$). For the rest of this section, any references to (\ref{FHN traveling wave ODE}) or its linearization (\ref{linearization around wave}) assume that the extra equation for $c'$ is included. Since $c$ is a center direction, the fixed point $0$ will have a three-dimensional center-unstable manifold $W^{cu}_\eps(0)$ and a three-dimensional center-stable manifold $W_\eps^{cs}(0)$, whose transverse intersection would be one-dimensional. Notice that the latter is still three-dimensional in the limit $\eps=0$, since the second stable direction becomes a center direction. We include the subscript $\eps$ to emphasize the dependence of these manifolds on $\eps$.

After appending the equation $c'=0$, the critical manifold is now two dimensional, parameterized by $v$ and $c$. If we think of the critical manifold as being the graph of a function $H(v,c)$, it is clear that its tangent space at a generic point $P$ is given by \begin{equation}
T_PM_0^{R/L}=\mathrm{sp}\left\{\p_v H,\p_c H\right\}=\mathrm{sp}\left\{\left[\begin{array}{c}
1/f'(u)\\
1\\
0\\
(1/c)(\gamma-1/f'(u))\\
0
\end{array}\right],\left[\begin{array}{c}
0\\
0\\
0\\
(-1/c^2)(\gamma v-u)\\
1
\end{array}\right]\right\}.
\end{equation} To highest order, the flow on $M_\eps^{L/R}$ for $\eps$ small or 0 will clearly be tangent to $\p_vH$, since $c$ is a parameter. This will be important later when we need to select a slow direction. For the sake of completeness, the flow on the critical manifold is now given by \begin{equation}\label{2d slow flow}
\dot{\left(\begin{array}{c}
v\\
c
\end{array}\right)}=\left(\begin{array}{c}
\frac{1}{c}(\gamma v-f^{-1}(v))\\
0
\end{array}\right).
\end{equation}

One of the hypotheses of the Exchange Lemma is that the shooting manifold $W^{cu}(0)$ transversely intersects $W^s(M_0^R)$ when $\eps=0$. We will show that this condition is satisfied in this section. To set things up, recall first that the two manifolds intersect along the fast jump, which we call \begin{equation}\label{fast jump label}
q_f(z)=(u(z),0,w(z),y(z),c^*).
\end{equation} To prove that this intersection is transverse, it suffices to check any one point along the orbit. A convenient place to check is very close to the landing point $p=(1,0,0,-1/c^*,c^*)$. We can therefore take $T_{q_f(z)}W^{s}(M_0^R)$ to be arbitrarily close to $T_pW^{cs}(p)$. This space is spanned by three vectors: the strong stable direction (which is the direction of the orbit), and the vectors $\p_v H$ and $\p_c H$, evaluated at $p$. $T_{q_f(z)}W^{cu}(0)$, on the other hand, is spanned by $q_f'(z)$, the invariant $y$ direction, and one more vector, which gives the change in $W^{cu}(0)$ as $c$ varies. To find this vector pick any point $q_f(z_0)$ on the fast jump. Since $c$ is a paramter (i.e. there is no flow in this direction), we can find a tangent vector $Y_0\in T_{q_f(z_0)}W^{cu}(0)$ of the form \begin{equation}\label{fast jump variation IC}
Y_0=(*,*,*,*,1).
\end{equation} We can then pick a curve $\alpha(c):(c^*-\delta,c^*+\delta)\rightarrow W^{cu}(0)$ such that $\alpha'(c^*)=Y_0$ and $\alpha(c^*)=q_f(z_0)$. By flowing the points on $\alpha(c)$ backwards in $z$, we obtain a one-parameter family of curves $\Gamma(z,c)$ in $W^{cu}(0)$. By construction, this family satisfies \begin{equation}
\p_c \Gamma(z,c)|_{c=c^*}\in T_{q_f(z)}W^{cu}(0)
\end{equation} for all $z\in\bbR$. Of interest then is the direction of the vector $\p_c\Gamma(z,c)|_{c=c^*}$ as $z\rightarrow\infty$ (i.e. as the jump approaches $p$). This is ascertained by observing that $\p_c\Gamma|_{c=c^*}$ satisfies the variational equation for (\ref{FHN traveling wave ODE}) along $q_f(z)$ with $\eps=0$. Indeed, using the equality of mixed partials, we have \begin{equation}
\p_z\left(\p_c \Gamma(z,c) \right)|_{c=c^*}=DF(\Gamma(z,c))\cdot\p_c\Gamma(z,c)|_{c=c^*}=DF(q_f(z))\cdot\left(\p_c\Gamma(z,c)\right)|_{c=c^*}.
\end{equation} Using the notation $\p_c\Gamma(z,c)|_{c=c^*}=(u_c(z),0,w_c(z),y_c(z),1)$, it follows that $W^{cu}(0)\pitchfork W^{s}(M_0^R)$ if and only if 
\begin{equation}\label{fast jump trans}
\det\left[\begin{array}{c c c c c}
u'(z) & u_c(z) & 0 & \frac{1}{a-1} & 0\\
0 & 0 & 0 & 1 & 0\\
w'(z) & w_c(z) & 0 & 0 & 0\\
y'(z) & y_c(z) & 1 & \frac{1}{c^*}\left(\gamma-\frac{1}{a-1}\right) & \left(1/c^*\right)^2\\
0 & 1 & 0 & 0 & 1
\end{array} \right]=u'(z)w_c(z)-w'(z)u_c(z)\neq 0
\end{equation} for $z\gg 1$. We are now prepared to prove transversality.

\begin{lemma}\label{F1 lemma}
The invariant manifolds $W_0^{cu}(0)$ and $W^s(M_0^R)$ intersect transversely along the fast jump $q_f(z)$.
\end{lemma}

\begin{proof}
Following the preceding discussion, the task is to show that $u'(z)w_c(z)-w'(z)u_c(z)\neq 0$ as $q_f(z)$ approaches the landing point $p$. Confirming this is a Melnikov-type calculation, which we verify using differential forms (cf. \S 4.3-4.5 of \cite{Jones_GSP}). Indeed this is natural, since the quantity of interest is $(du\wedge dw)$ applied to the first two columns of the matrix in (\ref{fast jump trans}). Derivatives of differential forms are computed in \cite{Jones_GSP} by relating them to the variational equation. The latter equation induces a derivation on exterior powers of $\bbR^n$, which is naturally dual to the space of differential forms. See \S 3 of \cite{CJ17} for more detail. Along the fast jump $q_f(z)$, one computes that \begin{equation}\label{two form deriv calc}
\begin{aligned}
(du\wedge dw)' & =du'\wedge dw+du\wedge dw' \\
& = dw\wedge dw+du\wedge(-f'(u)\,du-c^*dw-w\,dc)\\
& = -c^*du\wedge dw-w\,du\wedge dc.
\end{aligned}
\end{equation} The derivative in the preceding calculation refers to how the quantity $du\wedge dw$ changes when applied to two vectors evolving under the variational equation (\ref{linearization around wave}). In other words, it is the Lie derivative of the two-form $du\wedge dw$ along the vector field $F$, tangent to $q_f(z)$. Now set \begin{equation}
\alpha(z):=du\wedge dw(q_f'(z),\p_c\Gamma(z,c)|_{c=c^*})
\end{equation} Applying an integrating factor to (\ref{two form deriv calc}), we see that \begin{equation}
\p_z\left(e^{c^*z}\alpha\right)=-w^2e^{c^*z}.
\end{equation} Since $w\rightarrow 0$ faster than $e^{-c^*z}$ as $z\rightarrow-\infty$, this equation can be integrated to obtain \begin{equation}
\alpha(z)=-e^{-c^*z}\int_{-\infty}^{z}e^{c^*s}w^2\,ds,
\end{equation} from which it is clear that \begin{equation}
\lim\limits_{z\rightarrow\infty}e^{c^*z}\alpha=L<0.
\end{equation} The factor $e^{c^*z}$ ensures that the vectors $q_f'(z)$ and $\p_c\Gamma(z,c)|_{c=c^*}$ stay bounded and nonzero in the limit. The Lemma then follows, since it is the direction of these vectors (and not the magnitude) that is of interest.
\end{proof}
Notice that $\gamma$ played no roll in the result of this subsection. Accordingly, the following is a byproduct of the proof of the Lemma.
\begin{cor}
Assume that $\gamma>0$ is large enough so that $u=\gamma f(u)$ has three solutions $u_i$ satisfying $0=u_1<u_2<u_3$. Then for $\eps>0$ sufficiently small, (\ref{FHN traveling wave ODE}) possesses a heteroclinic orbit connecting the fixed points $(0,0,0,0)$ and $Q:=(u_3,u_3/\gamma,0,0)$ for $c=c^*+O(\eps)$. The heteroclinic orbit corresponds to a traveling front solution for (\ref{general PDE}) and is locally unique.
\end{cor}
\begin{proof}
It is clear that (\ref{FHN traveling wave ODE}) has three fixed points for the prescribed values of $\gamma$, and that $Q\in M_0^R$ for $\eps=0$. A heteroclinic connection exists between the two points if the unstable manifold of $0$ intersects the stable manifold of $Q$. From Fenichel theory \cite{Fen79}, the limit as $\eps\rightarrow 0$ of $W_\eps^{cs}(Q)$ is exactly $W^s(M_0^R)$, which we just proved intersects $W^{cu}(0)$ transversely. The transverse intersection perturbs to the case $\eps>0$, and the orbit in question is given by the intersection. The local uniqueness and dependence of the speed $c$ on $\eps$ are a consequence of the Implicit Function Theorem.
\end{proof} We remark that the framework of \cite{corn} and the calculation in \S 4 can be adapted to show that the traveling front just obtained is stable in the sense of Definition \ref{stability defn}. However, we will not pursue that further here.
\subsection{Transversality along the Back and Completion of the Construction}
Armed with an understanding of $W^{cu}(0)$ as it moves along the front, we now turn our attention to the passage near the slow manifold $M^R_\eps$ and the back. As explained in \S2.2, the Nagumo back is a heteroclinic connection between the jump-off point $q=(u^*,v^*,0,(1/c^*)(\gamma v^*-u^*),c^*)$ and the point \begin{equation}
\hat{q}=(u^*-1,v^*,0,(1/c^*)(\gamma v^*-u^*+1),c^*)
\end{equation} on $M_0^L$. In $uw-$space, the equations take the form \begin{equation}
u_b(z)=u^*-u_f(z), \hspace{.1 in}w_b(z)=-w_f(z),
\end{equation} where $u_f,w_f$ are the components of the front. The important facts for this section are that $u_b$ is monotonically decreasing, and $w_b$ decreases to a minimum and then increases from there. Now we focus on the first slow piece connecting the fast jumps, which involves an application of the Exchange Lemma.

To state and use the Exchange Lemma, we first rewrite the traveling wave equations (\ref{FHN traveling wave ODE}) in Fenichel coordinates, see \cite{JK94,Jones_GSP} for more details. In a neighborhood $B$ of the slow manifold $M_\eps^R$, we can change coordinates so that (\ref{FHN traveling wave ODE}) takes the form \begin{equation}\label{Fenichel coords}
\begin{aligned}
\mathbf{a}'&=\Lambda(\mathbf{a},\mathbf{b},\mathbf{y},\eps)\mathbf{a}\\
\mathbf{b}'&=\Gamma(\mathbf{a},\mathbf{b},\mathbf{y},\eps)\mathbf{b}\\
\mathbf{y}'&=\eps\left(U+G(\mathbf{a},\mathbf{b},\mathbf{y},\eps)\right),
\end{aligned}
\end{equation} with $\mathbf{a}\in\bbR^2,\mathbf{b}\in\bbR,\mathbf{y}=(y_1,y_2)\in\bbR^2$, and $U=(1,0)$. The region $B$ can be taken to be of the form \begin{equation}
B=\{(\mathbf{a},\mathbf{b},\mathbf{y}):|\mathbf{a}|<\delta,|\mathbf{b}|<\delta,\mathbf{y}\in K\},
\end{equation} where $\delta$ is small and $K$ is a compact set containing the landing point $p$ and jump-off point $q$ in its interior. On account of normal hyperbolicity of $M_0^R$, we know that for sufficiently small $\delta$ the eigenvalues of $\Lambda(0,0,\mathbf{y},0)$ are real, positive and uniformly bounded away from $0$. Likewise, $\Gamma(0,0,\mathbf{y},0)<C_\delta<0$. The function $G$ in (\ref{Fenichel coords}) is bilinear in $\mathbf{a},\mathbf{b}$ due to the fact that the sets $\mathbf{a}=0$ and $\mathbf{b}=0$ are invariant. The special form of the $\mathbf{y}$ component is obtained by rectifying the flow on the slow manifold. It is clear that for this problem $U$ is the ``straightened out" graph of the cubic for fixed $c$--that is, in the direction $\p_v H$--since there is no change in $c$ in the trajectory through any point.

The $(C^1)$ Exchange Lemma describes the configuration of a manifold of trajectories that spends a long time near $M_\eps^R$ before leaving the neighborhood $B$. The manifold of interest in our case is $W_\eps^{cu}(0)$, which the reader will recall is three-dimensional. (The subscript serves to emphasize the $\eps$-dependence.) The following statement of the Exchange Lemma is specialized to the setting of (\ref{FHN traveling wave ODE}). For the general statement and proof, the reader is directed to \cite{JK94}, or \cite{Kuehn15} for a sketch.
\begin{theorem}[``Exchange Lemma" of \cite{JK94}]\label{Exchange lma}
Assume that $W_0^{cu}(0)\pitchfork W^s(M_0^R)$. Let $J$ be a compact segment of the trajectory through $p$ for the limiting slow flow (\ref{2d slow flow}) that contains $q$. Then \begin{enumerate}
\item For any $r_0\in W^u_0(J)\cap \p B$, there exists $q_\eps\in W^{cu}_\eps(0)\cap\p B$ and a time $T_\eps>0$ such that $q_\eps\cdot T_\eps\in\p B$ and $|q_\eps\cdot T_\eps-r_0|=O(\eps).$ Furthermore, $T_\eps=O(\eps^{-1})$.
\vspace{.1 in}
\item Let $\bar{q}\in W_\eps^{cu}(0)\cap\{|\mathbf{a}|=\delta\}$ be the exit point of a trajectory through $q\in W_\eps^{cu}(0)\cap\{|\mathbf{b}|=\delta\}$ that spends time $T=O(\eps^{-1})$ in $B$. Let $V\subset W^{cu}_\eps(0)$ be a neighborhood of $q$. Then the image of $V$ under the time $T$ map is $O(\eps)$-close in $C^1$ norm to $W^c(J)$ in a neighborhood of $\bar{q}$.
\end{enumerate}
\end{theorem}
The first part of the theorem says that we can find points in $W_\eps^{cu}(0)$ near $p$ that pass by the slow manifold and then exit the neighborhood $B$ as close to the Nagumo back as we would like. The second part says that for such points, upon exiting the neighborhood $B$, the shooting manifold $W^{cu}_\eps(0)$ will be very close to the manifold $W^{u}(J)$. The fact that $W^{cu}(0)$ is crushed against an unstable manifold is to be expected. The strength of the Exchange Lemma lies in telling us which slow direction is picked out. (Recall that $W^u(M_\eps^R)$ is four-dimensional with the $c$ equation appended, so there is only room for one of the two slow directions.) The result is that the dominant slow direction is that of the trajectory connecting the landing point and jump-off point. On the level of tangent planes, we have \begin{equation}
T_{\bar{q}}W_\eps^{cu}(0)\approx T_q W^{u}(q)\oplus \mathrm{sp}\{\p_v H(q)\}.
\end{equation}

To complete the construction of the pulse, the final ingredient we need is that $W_0^{u}(J)$ intersects $W^{s}(M_0^L)$ transversely along the back. Indeed, the latter is the $\eps\rightarrow 0$ limit of $W_\eps^{cs}(0)$. If $W^s(M_0^L)\pitchfork W_0^u(J)$, then also $W^s(M_\eps^L)\pitchfork W_\eps^{cu}(0)$, since $W_\eps^{cu}(0)$ is $O(\eps)$ close to $W_0^u(J)$ by the Exchange Lemma. This is precisely what we need to show--that there is a (one-dimensional) transverse intersection between $W_\eps^{cu}(0)$ and $W_\eps^{cs}(0)$ for $\eps>0$ small. Since $c'=0$, the trajectory lying in the intersection therefore represents a homoclinic orbit to $0$ for (\ref{FHN traveling wave ODE}) with fixed $\eps$. The required transversality along the back is recorded in the following lemma.

\begin{lemma}\label{F2 lemma}
The invariant manifolds $W^{u}(J)$ and $W^s(M_0^R)$ intersect transversely along the second fast jump $q_b(z)$.
\end{lemma}

The proof is identical to that of Lemma \ref{F1 lemma}, so we omit the details. The reader is invited to check that it suffices to show that \begin{equation}\label{transversality 2}
\lim\limits_{z\rightarrow\infty}e^{c^*z}du\wedge dw(q_b'(z),\partial_v q_b(z))<0,
\end{equation} where $\p_v q_b(z)$--akin to $\p_c\Gamma(z,c)|_{c=c^*}$ from the front--gives the change in the orbit $q_b(z)$ as $v$ varies and is the unknown tangent direction to $W^u(M_0^R)$. The inequality (\ref{transversality 2}) is confirmed using the Melnikov integral \begin{equation}
\lim\limits_{z\rightarrow\infty}e^{c^*z}du\wedge dw({q}_b'(z),\partial_v {q}_b(z))=\int_{-\infty}^{\infty}e^{c^*z}{w}_b(z)\,dz.
\end{equation} As a remark, both transversality conditions are identical to those needed to construct the pulse for the 3D system (see \cite{JKL91} and \cite{KSS97}). The reason for this is that the extra (invariant) $y$ direction is unstable, so it will not be duplicated in the tangent space to $W^s(M_0^{R/L})$ at the respective landing points. This is readily seen from the matrix in (\ref{fast jump trans}).

Putting together Lemmas \ref{F1 lemma} and \ref{F2 lemma} with Theorem \ref{Exchange lma}, we can conclude the main result of this section.

\begin{theorem}\label{thm pulse existence}
For $\eps>0$ sufficiently small, equation (\ref{FHN traveling wave ODE}) possesses an orbit $\varphi_\eps$ homoclinic to $0$ for a wave speed $c(\eps)=c^*+O(\eps)$. Furthermore, $\varphi_\eps$ is $O(\eps)$ close to the singular orbit consisting of two alternating fast and slow segments. Finally, the orbit is locally unique.
\end{theorem}

\begin{proof}
We have already explained how the results of this section generate a transverse intersection of $W_\eps^{cs}(0)$ and $W_\eps^{cu}(0)$. The closeness to the singular orbit and the local uniqueness both follow from the Implicit Function Theorem, which is used to continue the transverse intersection to the $\eps\neq 0$ case.
\end{proof}

\section{Stability of the Wave}

Having established the existence of $\varphi$, we now turn to the issue of its stability. The stability problem for $\vp$ is discussed at length in \cite{corn}, so we refer the reader there for proofs and more details regarding the results in this section.

As discussed in the introduction, the stability of the wave is ultimately determined by the spectrum $\sigma(L)$ of the operator $L$ in (\ref{linearized operator}). The set $\sigma(L)$ can be decomposed into two parts. First, $\lambda\in\bbC$ is an eigenvalue of $L$ if there exists a solution $P=(p, q)^T\in BU(\bbR,\bbC^2)$ to the equation \begin{equation}\label{eval equation formal}
LP=\lambda P.
\end{equation} The set of eigenvalues of $L$ of finite multiplicity is denoted $\sigma_n(L)$. Setting $p_z=r$ and $q_z=\eps s$, we can write the eigenvalue problem as a first order system \begin{equation}\label{eval eqn 1os}
\left(\begin{array}{c}
p\\q\\r\\s
\end{array} \right)_z=\left(\begin{array}{c c c c}
0 & 0 & 1 & 0\\
0 & 0 & 0 & \eps\\
\lambda-f'(\hat{u}) & 1 & -c & 0\\
-1 & \frac{\lambda}{\eps}+\gamma & 0 & -c
\end{array}\right)\left(\begin{array}{c}
p\\q\\r\\s
\end{array}\right).
\end{equation} This nonautonomous, linear system is abbreviated \begin{equation}\label{eval eqn abbr}
Y'(z)=A(\lambda,z)Y(z).
\end{equation} Notice that the only $z$-dependence of $A(\lambda,z)$ is from $\hat{u}$, which we know decays exponentially to $0$ as $z\rightarrow\pm\infty$. We therefore have a well-defined limit \begin{equation}\label{A_infinity(lambda) defn}
A_\infty(\lambda)=\lim\limits_{z\rightarrow\pm\infty}A(\lambda,z)=\left(\begin{array}{c c c c}
0 & 0 & 1 & 0\\
0 & 0 & 0 & \eps\\
\lambda+a & 1 & -c & 0\\
-1 & \frac{\lambda}{\eps}+\gamma & 0 & -c
\end{array}\right).
\end{equation} 

The set of eigenvalues is only part of the spectrum of $L$. The rest is the so-called essential spectrum $\sigma_\mathrm{ess}(L)$, which in this case is given by (cf. Lemma 3.1.10 of \cite{KP13}) \begin{equation}\label{ess spec defn}
\sigma_\mathrm{ess}(L)=\{\lambda\in\bbC:A_\infty(\lambda) \text{ has an eigenvalue }\mu\in i\bbR \}.
\end{equation}
It is shown in Lemma 1 of \cite{CJ17} that $\sigma_\mathrm{ess}(L)$ is contained in a half-plane of the form \begin{equation}\label{ess spec half plane}
\mathcal{K}=\{\lambda\in\bbC:\mathrm{Re}\,\lambda<K \},
\end{equation} for some $K<0$. $K$ cannot be chosen independently of $\eps$, but this is not a problem since $\eps$ is fixed in the stability analysis. We therefore avoid many of the difficulties encountered in \cite{Jones84}. Although there is in general a disconnect between spectral, linear, and nonlinear stability of solitons, for systems of the form (\ref{general PDE}), spectral stability is sufficient, cf. \cite{BJ89,Henry}.

\begin{theorem}\label{nonlinear stab theorem}
Suppose that the operator $L$ satisfies \begin{enumerate}
\item There exists $\beta<0$ such that $\sigma(L)\setminus\{0\}\subset\{\lambda\in\bbC:\mathrm{Re }\lambda<\beta\}$.
\item $0$ is a simple eigenvalue.
\end{enumerate} Then $\varphi$ is stable in the sense of Definition \ref{stability defn}.
\end{theorem}
Note that the translation invariance mentioned above forces $0$ to be an eigenvalue of $L$. Via the Evans function, its algebraic multiplicity is shown in \cite{AJ94} (pp. 57-60) to be one if the wave is transversely constructed. This is exactly what we proved in \S 3. We remark that this also follows from Lemma 3 of \cite{CJ17}, together with Theorem 5.2 of \cite{corn}. To prove stability, it therefore suffices to show that (1) holds in Theorem \ref{nonlinear stab theorem}. A big step in this direction is the following, which is proved in \S 5.1 of \cite{corn}.

\begin{lemma}[Lemma 5.2 of \cite{corn}]\label{realness of spectrum}
For $\eps>0$ sufficiently small, if $\lambda\in\sigma_n(L)\cap(\bbC\setminus\mathcal{K})$ and $\mathrm{Re}\,\lambda\geq-\frac{c^2}{8}$, then $\lambda\in\bbR$.
\end{lemma} We therefore see that the conditions of Theorem \ref{nonlinear stab theorem} are satisfied as long as $L$ has no real, positive eigenvalues. Indeed, since $0\in\sigma_n(L)$ is a simple eigenvalue, there is an interval $(-\delta,\delta)\subset\bbR$ containing no eigenvalues of $L$. Taking $\beta$ to be the maximum of $-\delta$, $-c^2/8$, and $K$ from (\ref{ess spec half plane}), it follows that the only non-zero eigenvalues with real part greater than $\beta$ must be real and positive. The Maslov index can be used to detect these unstable eigenvalues.
 \subsection{The Maslov Index} Let $\langle\cdot,\cdot\rangle$ denote the standard dot product on $\bbR^4$. We define a \emph{complex structure} on $\bbR^4$ by the matrix \begin{equation}\label{matrix J}
 J=\left(\begin{array}{c c c c}
0 & 0 & 1 & 0\\0 & 0 & 0 & -1\\-1 & 0 & 0 & 0\\0 & 1 & 0 & 0
 \end{array}\right).
 \end{equation} It is a standard fact (cf. \S 1 of \cite{Heck13}) that $J$ and $\langle\cdot,\cdot\rangle$ define a symplectic (i.e. skew-symmetric, nondegenerate, bilinear) form on $\bbR^4$ by the formula \begin{equation}\label{omega defn}
 \omega(a,b)=\langle a,Jb\rangle.
 \end{equation} The key to exploiting this fact for the stability analysis of $\vp$ is that the value of this form can be tracked on any two solutions of (\ref{eval eqn 1os}), for $\lambda\in\bbR$ fixed. It is shown in \cite{corn,CJ17} that: \begin{theorem}\label{thm Lambda n invariant}
 	Let $u,v$ be two solutions of (\ref{eval eqn 1os}) for fixed $\lambda\in\bbR$. Then \begin{equation}
 	\frac{d}{dz}\omega(u,v)=-c\omega(u,v).
 	\end{equation} Consequently, the form $\Omega(\cdot,\cdot):=e^{cz}\omega(\cdot,\cdot)$ is independent of $z$.
 \end{theorem} For a proof, the reader is referred to pages 11-12 of \cite{CJ17}. It follows from Theorem \ref{thm Lambda n invariant} that if $\omega(u(z),v(z))=0$ for any $z\in\bbR$, then $\omega(u,v)\equiv 0$. A plane $V\in\Gr_2(\bbR^4)$ is called \emph{Lagrangian} if \begin{equation}
\omega|_{V}\equiv 0.
\end{equation} The set of Lagrangian planes in $\bbR^4$ is a three-dimensional compact manifold, called the Lagrangian Grassmannian $\Lambda(2)$. Now, since (\ref{eval eqn 1os}) is linear, it induces an equation on $\Gr_k(\bbR^4)$ for each $k$. In particular, the following is a corollary of Theorem \ref{thm Lambda n invariant}. \begin{theorem} The Lagrangian Grassmannian $\Lambda(2)$ is an invariant manifold for the equation induced by (\ref{eval eqn 1os}) on $\Gr_2(\bbR^4)$.
\end{theorem} The case $k=2$ is of interest because the \emph{stable} and \emph{unstable bundles} are two-dimensional. Indeed, $A_\infty(0)$ has two positive and two negative eigenvalues by (\ref{evals of lin}), and the eigenvalue split can only change across the essential spectrum. It follows that $A_\infty(\lambda)$ has two eigenvalues each of positive and negative real part for all $\lambda$ with $\mathrm{Re}\,\lambda\geq\beta$. It is then standard (cf. Theorem 3.2 of \cite{Sandstede02}) that (\ref{eval eqn 1os}) admits exponential dichotomies on $\bbR^+$ and $\bbR^-$ with the same Morse index, and we can define \begin{equation}\label{(un)stable bundles}
\begin{aligned}
E^u(\lambda,z) & = \{\xi(z)\in\bbC^{2n}:\xi \text{ solves } (\ref{eval eqn 1os}) \text{ and } \xi\rightarrow 0 \text{ as }z\rightarrow -\infty \}\\
E^s(\lambda,z) & = \{\xi(z)\in\bbC^{2n}:\xi \text{ solves } (\ref{eval eqn 1os}) \text{ and } \xi\rightarrow 0 \text{ as }z\rightarrow \infty \}
\end{aligned}.
\end{equation} These spaces vary analytically in $\lambda$ and contain all of the solutions of (\ref{eval eqn 1os}) that decay as $z\rightarrow-\infty$ ($E^u(\lambda,z)$) or as $z\rightarrow\infty$ ($E^s(\lambda,z)$). Furthermore, the decay of any solution is exponential in $z$. It follows that $\lambda\in\bbC$ is an eigenvalue of $L$ if and only if \begin{equation}
E^u(\lambda,z)\cap E^s(\lambda,z)\neq\{0\}
\end{equation} for some (and hence all) $z\in\bbR$. The following crucial fact is proved in \cite{CJ17}. \begin{theorem}[Theorem 1 of \cite{CJ17}]\label{un/stable bundles Lagrangian} For each $\lambda\in(\bbR^+\cup\{0\})$, $E^{u/s}(\lambda,z)$ are Lagrangian subspaces for all $z\in\bbR$. 
\end{theorem} $E^{u/s}(\lambda,z)$ therefore each define two-parameter curves in $\Lambda(2)$. The Maslov index \cite{Maslov,Arnold67,Arnold85} counts how many times such a curve intersects a particular hypersurface in $\Lambda(2)$. For a fixed plane $V\in\Lambda(2)$, the \emph{train} of $V$ is defined to be \begin{equation}\label{train defn}
\Sigma(V)=\{V'\in\Lambda(2):\dim(V\cap V')>0 \}.
\end{equation} There is a clear partition of this set $\Sigma(V)=\Sigma_1(V)\cup\Sigma_2(V)$, with \begin{equation}
\Sigma_i(V)=\{V\in\Lambda(2):\dim(V\cap V')=i\}.
\end{equation} In particular, $\Sigma_2(V)=\{V\}$. It is shown in \S 2 of \cite{Arnold67} that $\overline{\Sigma_1(V)}=\Sigma(V)$, and $\Sigma_1(V)$ is an oriented codimension-one submanifold of $\Lambda(2)$. Thus the Maslov index can be defined for any curve $\gamma:[a,b]\rightarrow\Lambda(2)$ which only intersects $\Sigma(V)$ through $\Sigma_1(V)$ to be the signed count of intersections with $\Sigma_1(V)$. In \cite{RS93}, the definition of the Maslov index was expanded to include curves that intersect $\Sigma(V)$ in any stratum. The key was to make precise the notion of a curve intersecting $\Sigma(V)$, which was accomplished through the introduction of the ``crossing form." This is a quadratic form whose signature determines the contribution to the Maslov index at each crossing, called \emph{conjugate points}. For a quadratic form $Q$, we denote by $n_+(Q)$ and $n_-(Q)$, respectively, the positive and negative indicies of inertia of $Q$ (see page 187 of \cite{vinberg}). The signature of $Q$ is then defined by \begin{equation}\label{sig defn}
\mathrm{sign}(Q)=n_+(Q)-n_-(Q).
\end{equation} As mentioned above, for a curve $\gamma:[a,b]\rightarrow\Lambda(2)$ parametrized by $t$, a value $t=t^*$ such that $\gamma(t^*)\cap V\neq\{0\}$ is called a conjugate point. A conjugate point is called \emph{regular} if the crossing form--defined on the intersection $\gamma(t^*)\cap V$--is nondegenerate. 

Rather than the define the crossing form abstractly (see \cite{RS93} or \cite{CJ17}), we will focus directly on the problem at hand. To define crossings (and hence the Maslov index), one needs a curve and a reference plane. The curve we will consider is the unstable bundle $E^u(0,z)$. For technical reasons explained in \S 5 of \cite{CJ17}, the domain of the curve $z\mapsto E^u(0,z)$ is taken to be $(-\infty,\tau]$, where $\tau$ is chosen large enough so that \begin{equation}\label{tau req}
E^s(0,z)\cap V^u(0)=\{0\} \text{ for all } z\geq\tau.
\end{equation} The reference plane is then taken to be $E^s(0,\tau)$, for the same value $\tau$. It follows that there is a conjugate point at $z=\tau$, since $\vp'(\tau)\in E^u(0,\tau)\cap E^s(0,\tau)$. This conjugate point encodes the translation invariance of $\vp$, and it is shown in \cite{CJ17} that it plays a distinguished role in the stability analysis of $\vp$. Intuitively, we think of the curve $z\mapsto E^u(0,z)$ as shooting the ``left boundary data" forward and counting intersections with the ``right boundary data" $E^s(0,\tau)$. In this sense, the Maslov index is very much a generalization of Sturm-Liouville theory on an interval. Let $z=z^*$ be a conjugate point for the unstable bundle. The \emph{crossing form} is defined by \begin{equation}\label{xing form defn}
\Gamma(E^u(0,\cdot),E^s(0,\tau),z^*)(\xi)=\omega(\xi,A(0,z^*)\xi),
\end{equation} for $\xi\in E^u(0,z^*)\cap E^s(0,\tau)$, and $A(0,z^*)$ as in (\ref{eval eqn abbr}). This form is derived in Theorem 3 of \cite{CJ17}. We can then define the Maslov index of the traveling wave $\vp$ as follows. \begin{define}\label{Maslov defn}
	Let $\tau\gg 1$ satisfy (\ref{tau req}). The \textbf{Maslov index} of $\vp$ is given by \begin{equation}
	\Maslov:=\sum_{z^*\in(-\infty,\tau)}\mathrm{sign}\Gamma(E^u(0,\cdot),E^s(0,\tau),z^*)+n_+(\Gamma(E^u(0,\cdot),E^s(0,\tau),\tau),
	\end{equation} where the sum is taken over all interior crossings of $E^u(0,z)$ with $\Sigma(E^s(0,\tau)).$
\end{define} It is proved in \S 1 of  \cite{CH} that this definition is independent of $\tau$, provided that (\ref{tau req}) is satisfied.

Two remarks about this definition are in order. First, the Maslov index for paths with distinct endpoints can only be defined if all crossings are regular, see \cite{RS93}. However, this is not an issue here, since irregular crossings (i.e. those for which $\Gamma$ is degenerate) are non-generic (\S 2 of \cite{Arnold67}). We can therefore perturb away from them by changing $\tau$, which clearly moves the train $E^s(0,\tau)$, but not the image of the curve $E^u(0,z)$ (other than its right endpoint). Since the Maslov index definition is independent of $\tau$, we can rest assured that all crossings are regular. Moreover, the calculation of $\Maslov$ in the next section detects all conjugate points, and we calculate directly from (\ref{xing form defn}) that the crossing is regular in each case.

Second, we wish to justify the choice of adding $n_+(\Gamma)$ for the conjugate point $z=\tau$. This is actually different from the convention of \cite{RS93}, in which $1/2$ times the signature of each endpoint crossing form is added to the Maslov index. We prefer to follow the convention of \cite{HLS16} to ensure that $\Maslov$ is an integer. The endpoint contribution to the Maslov index is merely convention, as long as the index remains additive vis-\`{a}-vis concatenation of curves. For example, our choice of convention demands that $-n_-(\Gamma)$ is used as the contribution at a left endpoint crossing, so that $\mathrm{sign}\,\Gamma$ is recovered if two curves are concatenated. This left endpoint term does not appear in Definition \ref{Maslov defn} because there is no left endpoint crossing; $E^u(0,-\infty)=V^u(0)$, which is transverse to $E^s(0,\tau)$, by (\ref{tau req}). The choice of $n_+(\Gamma)$ over $-n_-(\Gamma)$ for the right endpoint is made so that the following theorem holds. \begin{theorem}[Theorem 5.1 of \cite{corn}]\label{Morse = Maslov thm} Define the Morse index $\mathrm{Mor}(L)=|\sigma(L)\cap\{\lambda\in\bbC:\mathrm{Re}\,\lambda\geq 0\}|$ to be the number of unstable eigenvalues of $L$, counted with algebraic multiplicity. Then 
	\begin{equation}
	\Maslov=\Mor(L).
	\end{equation} 
\end{theorem} The proof of this theorem is given in \S 5 of \cite{corn}. What is nice about this result is that the spectral information needed to prove that $\vp$ is stable is contained entirely in the variational equation for (\ref{FHN traveling wave ODE}) along $\vp$, since $\lambda=0$ in the calculation of $\Maslov$. In the next section, we carry out the calculation showing that $\Maslov=0$. This, in turn, proves the main result of this work. \begin{theorem}\label{thm stability of wave} For $\eps>0$ sufficiently small, the traveling waves $\vp_\eps(z)$ guaranteed to exist by Theorem \ref{thm pulse existence} are stable in the sense of Definition \ref{stability defn}.
\end{theorem}

\section{Calculating the Maslov Index}
Recall that $\Maslov$ is calculated by following the curve $E^u(0,z)$ from $z=-\infty$ to $z=\tau$. Solving (\ref{linearization around wave}) directly (thus determining the curve of interest) is a tall order, since that equation is nonautonomous and dependent on $\eps$. Instead, we will take advantage of the well-known fact that $E^u(0,z)$ is tangent to $W^u(0)$ along $\vp$ (cf. \S 6 of \cite{CJ17}). This manifold can be followed around phase space using the timescale separation, which makes the calculation of the index tractable. The strategy is to inspect each fast and slow piece of $\vp$ separately, as well as the transitions between them. We will show in this section that each of the four segments contains one one-dimensional crossing, two of which are positive and two of which are negative. Adding these together gives $\Maslov=0$, which proves Theorem \ref{thm stability of wave}. 

Just like $E^u(0,z)$ is tangent to $W^u(0)$ along $\vp$, so is $E^s(0,z)$ tangent to $W^s(0)$. It follows that the reference plane $E^s(0,\tau)$ is given by $T_{\vp(z)}W^s(0)$, where $\vp(z)$ is as close as we like to returning to $0$. By Fenichel theory, this subspace is spanned (to leading order) by the tangent vector to $M_\eps^L$ and the stable eigenvector of the same point on the critical manifold. We label the components of this point $\vp(\tau)=(u_\tau,v_\tau,w_\tau,y_\tau)$. Throughout the calculation, we will make heavy use of the robustness of transverse intersections. More precisely, the train of $E^s(0,\tau)$ is a codimension one subset of $\Lambda(2)$. If the curve $E^u(0,z)$ crosses it transversely for some value $z=z^*$, then the crossing would persist for sufficiently small perturbations of both the curve and the reference plane. We are therefore justified in taking the leading order approximations of both $E^u(0,z)$ and $E^s(0,\tau)$. This allows us to search for intersections on the fast and slow timescales with $\eps=0$, which is significantly easier. In particular, we can use the singular value $c=c^*$ throughout. We drop the $^*$ for the rest of the section.

The most difficult part of the calculation is proving the nonexistence of conjugate points near the three corners where transitions between fast and slow dynamics occur. Near these points, the $\eps\rightarrow 0$ limit of the curve $T_{\vp(z)}W^u(0)\subset\Lambda(2)$ has jump discontinuities, so we must figure out how these gaps are bridged when $\eps>0$ is small. This is accomplished by analyzing the flow induced on $\Lambda(2)$ by a constant coefficient linear system. The phase portrait for such an equation is completely understood, and the relevant details are recorded in Appendix B. In Appendix A, we give a brief overview of Pl\"{u}cker coordinates, which are used to write down the equation that is induced by (\ref{linearization around wave}) on $\Lambda(2)$. Indeed, (\ref{linearization around wave}) induces a derivation on $\bigwedge^2\bbR^4$, and the Pl\"{u}cker embedding allows us to realize points in $\Lambda(2)$ as elements of $\bigwedge^2\bbR^4$. For convenience of the reader, we partition this section into subsections wherein each piece of the wave is considered separately.

\subsection{First Fast Jump}
As explained in \S 2.2, we have a very clear picture of the unstable bundle (i.e. of $T_{\vp(z)}W^u(0)$) along the fast front; at each point along the orbit, it is $O(\eps)$ close to $W^u(0_f)$, the unstable manifold for $0$ in (\ref{FHN layer problem}). The latter is just a cylinder over the Nagumo front, so its tangent space at any point along the jump is known. In anticipation of computing the crossing form, we include the $\delta v$ component, even though it will be $0$ for both basis vectors. We can differentiate (\ref{fast jump profile}) with respect to $u$ to determine one vector tangent to $W^u(0)$, and the other is given by the invariant $y$ direction. We therefore have \begin{equation}\label{fast front tan space}
T_{\vp(z)}W^u(0)\approx\mathrm{sp}\left\{\left[\begin{array}{c}
1\\
0\\
\sqrt{2}/2-u\sqrt{2}\\
0
\end{array}\right],\left[\begin{array}{c}
0\\0\\0\\1
\end{array}\right] \right\}.
\end{equation} To detect conjugate points, we need a working basis for $E^s(0,\tau)$. In light of the discussion at the beginning of this section, one basis vector is found by differentiating the equation defining $M_0$ in (\ref{critical mfd}) with respect to $v$. The other is computed by finding the stable eigenvector for the linearization of (\ref{FHN layer problem}) around $\vp(\tau)$ with $\eps=0$. It is then a calculation to see that \begin{equation}\label{ref plane}
E^s(0,\tau)\approx \mathrm{sp}\left\{\left[\begin{array}{c}
1\\f'(u_\tau)\\0\\\frac{1}{c}(\gamma f'(u_\tau)-1)
\end{array}\right],\left[\begin{array}{c}
f'(u_\tau)\\0\\f'(u_\tau)\mu_1(u_\tau)\\ \mu_1(u_\tau)
\end{array}\right] \right\},
\end{equation} where $\mu_1(u_\tau)$ is the stable eigenvalue for the linearization of (\ref{FHN layer problem}) about $\vp(\tau)$. (Notice that the eigenvalues and eigenvectors for points in $M_0^L$ can be written as a function of $u$.) We use $\approx$ to remind the reader that this the leading order (in $\eps$) approximation to $E^s(0,\tau)$. Although we could use (\ref{ref plane}) directly to find conjugate points, the calculation would be tedious due to the way $\mu_1(u_\tau)$ depends on $u$. Instead, we claim that the Maslov index contribution is the same if we instead look for intersections with the train of \begin{equation}\label{stable subspace basis}
V^s(0)=\mathrm{sp}\left\{\left[\begin{array}{c}
1\\-a\\0\\ \frac{1}{c}(1+\gamma a)
\end{array}\right],\left[\begin{array}{c}
1\\0\\-a\sqrt{2}\\ \sqrt{2}
\end{array}\right] \right\},
\end{equation} which one obtains by substituting $u=0$ in (\ref{ref plane}) and using (\ref{singular c value}) and (\ref{layer lin 0 evals}). Indeed, we know that $u_\tau\rightarrow 0$ as $\tau\rightarrow\infty$, so that $\Sigma(E^s(0,\tau))$ will be very close to $\Sigma(V^s(0))$ as long as $\tau$ is large enough. Thus, as long any crossings of $E^u(0,z)$ with $V^s(0)$ are one-dimensional and transverse, then $E^u(0,z)$ would have to cross $E^s(0,\tau)$ nearby and in the same direction. 

Using this new reference plane, we see from (\ref{fast front tan space}) and (\ref{stable subspace basis}) that an intersection occurs if and only if \begin{equation}\label{fast front xing condition}
u=a+\frac{1}{2}.
\end{equation}
Since $u$ increases monotonically along the fast jump from $0$ to $1$, it follows that there is a unique conjugate point, and the intersection is spanned by $\xi:=[1,0,-a\sqrt{2},\sqrt{2}]^T$. To determine the direction of the crossing, we evaluate $\Gamma$ from (\ref{xing form defn}) on this vector. We call the conjugate point $z^*$, and use (\ref{omega defn}) to compute: \begin{equation}\label{F1 crossing form}
\begin{aligned}
\omega(\xi,A(0,z^*)\xi)=\langle \xi,JA(0,z^*)\xi\rangle & =-f'(u)+ca\sqrt{2}-2a^2\\
 & = -f'\left(\frac{1}{2}+a\right)+ca\sqrt{2}-2a^2\\
 & = a^2-\frac{1}{4}<0.
\end{aligned}
\end{equation} This shows that the crossing is negative, and we conclude that the contribution to $\Maslov$ is $-1$ along the fast front.

\subsection{First Corner}
Near the first landing point $p=(1,0,0,-1/c)$, the shooting manifold will undergo an abrupt reorientation. As the front approaches $p$, the tangent space to the shooting manifold will be spanned (approximately) by the stable eigenvector of the fixed point $(1,0,-1/c)$ for the fast subsystem and the invariant $y$ direction, as in the previous subsection. Combining (\ref{layer lin 1 evals}) with the observation $f'(1)=a-1$ and the calculation (\ref{ref plane}), we see that  \begin{equation}\label{corner 1 in}
T_{p_\mathrm{in}}W^u(0)\approx\mathrm{sp}\{\eta_1(p),\eta_3(p)\}=\mathrm{sp}\left\{\left[\begin{array}{c}
a-1\\
0\\
(1-a)\frac{\sqrt{2}}{2}\\
-\frac{\sqrt{2}}{2}
\end{array}\right],\left[\begin{array}{c}
0\\0\\0\\1
\end{array}\right] \right\}.
\end{equation}
The subscript ``in'' on $p$ refers to the fact that this is the tangent space to $W^u(0)$ upon entrance into a neighborhood of $p$, as opposed to trip away from $p$, up the slow manifold. The notation $\eta_i(p)$ indicates that the corresponding vector is an eigenvector for (\ref{layer lin 1}) with eigenvalue $\mu_i(p)$.

The next task is to determine the configuration of $W^u(0)$ as it moves up the slow manifold $M_\eps^R$. For this part of the journey, the derivative of the wave is given (to leading order) by the tangent vector to $M_\eps^R$. At $p$, this corresponds to the $0$-eigenvector \begin{equation}
\eta_2(p)=\left[\begin{array}{c}
1\\a-1\\0\\ \frac{1}{c}\left(\gamma(a-1)-1\right)
\end{array}\right].
\end{equation} It is less obvious which is the second direction picked out. Deng's Lemma \cite{Deng89,Schecter08} asserts that $W^u(0)$ will be crushed against $W^u(M_\eps^R)$, the unstable manifold of the right slow manifold. However, there are two unstable directions for each point on the critical manifold, and it is unclear which of these is picked out. (Since the approach to $p$ was in the weak unstable direction, it is not unreasonable to think that this direction would persist.) Thankfully, the symplectic structure is able to break the tie. We know from Theorem \ref{un/stable bundles Lagrangian} that $T_{\vp(z)}W^u(0)$ is a Lagrangian subspace of $\bbR^4$. Since $\Lambda(2)$ is closed in $\Gr_2(\bbR^4)$, the symplectic form $\omega$ must vanish on the leading order approximation to $T_{\vp(z)}W^u(0)$ as well. A direct computation shows that \begin{equation}
\omega(\eta_2(p),\eta_3(p))=1-a \neq 0,
\end{equation} so it must be that \begin{equation}
T_{p_\mathrm{out}}W^u(0)\approx\mathrm{sp}\{\eta_2(p),\eta_4(p) \}=\mathrm{sp}\left\{\left[\begin{array}{c}
1\\a-1\\0\\ \frac{1}{c}\left(\gamma(a-1)-1\right)
\end{array}\right],\left[\begin{array}{c}
1\\
0\\
\sqrt{2}(1-a)\\
-\sqrt{2}
\end{array}\right] \right\}.
\end{equation} 

Again, the vector $\eta_4(p)$ is obtained by using the formula for a generic eigenvector from (\ref{ref plane}) in conjunction with (\ref{layer lin 1 evals}). For the rest of this subsection, we will write $\eta_i$ instead of $\eta_i(p)$. The goal is to show that there are no conjugate points during the transition from $p_\mathrm{in}$ to $p_\mathrm{out}$.  Although this appears to be an $\eps\neq 0$ consideration, it is actually understood by analyzing the constant coefficient linear system obtained by setting $u\equiv 1$ in (\ref{linearization around wave}) with $\eps=0$. Indeed, by taking $\eps$ very small, we can ensure that $W^u(0)$ is as close to $p$ as desired while still maintaining (approximately) the shape of the cylinder. Similarly, since the traveling wave is $C^1$ $O(\eps)$-close to the singular object, the slow (i.e. tangent) direction is picked up arbitrarily close to $p$ on $M^R_\eps$. The second tangent vector is a solution to the linearized equation (\ref{linearization around wave}) with initial condition close to $\eta_3$, which is already an unstable direction. It follows that this direction must remain close to the unstable subspace of $p$, since the wave itself stays arbitrarily close to $p$ during the transition, and the unstable subspace of (\ref{FHN layer problem}) at $p$ is invariant. The above discussion of the symplectic structure then implies that this solution must be bumped to $\eta_4$, the strong unstable direction, during this transition.

Setting $X_{ij}=\mathrm{sp}\{\eta_1,\eta_j\}$ (see the appendix), we therefore must solve a pseudo-boundary value problem to connect the points $T_{p_\mathrm{in}}W^u(0)=X_{13}$ and $T_{p_\mathrm{out}}W^u(0)=X_{24}$ in $\Lambda(2)$ for the equation induced by \begin{equation}\label{corner1 linear problem}
\left(\begin{array}{c}
p\\
q\\
r\\
s
\end{array}\right)'=\left(\begin{array}{c c c c }
0 & 0 & 1 & 0\\
0 & 0 & 0 & 0\\
1-a & 1 & -c & 0\\
-1 & \gamma & 0 & -c
\end{array}\right)\left(\begin{array}{c}
p\\q\\r\\s
\end{array}\right),
\end{equation} which is (\ref{linearization around wave}) evaluated at $\hat{u}=1$. These two points are both equilibria for said equation, since the eigenspaces of the matrix in (\ref{corner1 linear problem}) are invariant. It follows that the desired connection must be a heteroclinic orbit. It is explained in Appendix B that $W^u(X_{13})\cap W^s(X_{24})$ is one-dimensional, so it suffices to find a single point in each distinct orbit to describe the intersection completely. The Schubert cell description of $W^{s/u}(X_{ij})$ makes it easy to see that there are two distinct heteroclinic connections from $X_{13}$ to $X_{24}$. These orbits--call them $\gamma_{\pm}$--pass through the points \begin{equation}\label{corner1 initial}
W_{\pm}=\mathrm{sp}\{\eta_1\pm\eta_2,\eta_3\pm k\eta_4 \}.
\end{equation} The constant \begin{equation}
k=-\frac{\omega(\eta_2,\eta_3)}{\omega(\eta_1,\eta_4)}=\frac{\sqrt{2}}{3-2a}>0
\end{equation} is needed to ensure that the planes $W_{\pm}$ are Lagrangian. This restriction is very beneficial; were we looking for the same connections in the full Grassmannian, then there would be a two-dimensional set of orbits indexed by $k(\neq 0)$. Now, to prove that there is no contribution to the Maslov index near the corner, it suffices show that the trajectory through $W_{\pm}$ is disjoint from $\Sigma(V^s(0))$. Since the ``boundary data'' for this equation are given in terms of the basis of eigenvectors at $u=1$, the easiest way to describe the solution is to use this basis for the Pl\"{u}cker coordinates as well. (See Appendix A for a discussion of Pl\"{u}cker coordinates.) The drawback is that the reference plane $V^s(0)$ must be rewritten in terms of this new basis, which can be done with the help of Maple: \begin{equation}\label{stable subspace corner1 basis}
\begin{aligned}
V^s(0) & =\mathrm{sp}\left\{\nu_1,\nu_2 \right\},\\
\nu_1 & = -2\eta_1+2c(2a-3)\eta_3+(2a-1)(a-1)\eta_4  \\
\nu_2 & = 2(2a-1)\eta_1+a(3-2a)\eta_2+\frac{(2a-1)(2a-3)}{-c}\eta_3+(1-2a)\eta_4.
\end{aligned}
\end{equation}

The heteroclinic orbit in $\Lambda(2)$ through $W_{\pm}$ is the projectivized version of the solution to the equation induced by (\ref{corner1 linear problem}) on $\bigwedge^2(\bbR^4)$ with initial condition \begin{equation}\label{corner1 Plucker initial}
\tilde{W}_\pm=(\eta_1\pm\eta_2)\wedge(\eta_3\pm k\eta_4)=(0,1,\pm k,\pm 1,k,0).
\end{equation} The ordered 6-tuple in (\ref{corner1 Plucker initial}) gives the Pl\"{u}cker coordinates of $W_\pm$ in the new basis. We can now give the explicit solution through this point, since the $\eta_i$ are eigenvectors for the matrix in (\ref{corner1 linear problem}): \begin{equation}
\gamma_\pm(z)=(0,e^{(\mu_1+\mu_3)z},\pm ke^{(\mu_1+\mu_4)z},\pm e^{(\mu_2+\mu_3)z},ke^{(\mu_2+\mu_4)z},0).
\end{equation} Since these coordinates are projective, we can divide by $e^{-cz}=e^{(\mu_1+\mu_4)z}=e^{(\mu_2+\mu_3)z}$ to obtain a more tractable representation of the same planes, \begin{equation}\label{corner1 solution}
\tilde{\gamma}_\pm(z)=(0,e^{-\frac{\sqrt{2}}{2}z},\pm k,\pm 1,ke^{\frac{\sqrt{2}}{2}z},0).
\end{equation} We claim that it suffices to show that $\tilde{\gamma}_+$ does not cross $\Sigma(V^s(0))$. Indeed, consider the concatenated curve $\gamma_0:=\tilde{\gamma}_+*-\tilde{\gamma}_-$, which is a loop in $\Lambda(2)$. It is known that the Maslov index of a loop does not depend on the choice of reference plane, since the Maslov index can be interpreted as an element in the cohomology group $H^1(\Lambda(2),\bbZ)$ \cite{Arnold67,Arnold85,Duis76}. Taking the reference plane to be $V=\mathrm{sp}\{\eta_3,\eta_4\}=(0,0,0,0,0,1)$, it follows from (\ref{Plucker xing condition}) that crossings are given by the equation $p_{12}=0$. For $\gamma_0(z)$, $p_{12}\equiv 0$, so $\gamma_0(z)$ is entirely contained in $\Sigma(V)$. However, the plane $V$ itself is not in the image of $\gamma_0$, which means that $\dim(\gamma_0(z)\cap V)\equiv 1$. It then follows from Theorem 2.3 (Zero) of \cite{RS93} that the Maslov index of $\gamma_0$ is $0$. Since the Maslov index is additive by concatenation (Theorem 2.3 (Catenation) of \cite{RS93}), it follows that the Maslov indices of $\gamma_+$ and $\gamma_-$ with respect to any reference plane are opposite of each other. We now show that $\gamma_+$ has no crossings with $\Sigma(V^s(0))$, which proves that there is no contribution to the Maslov index at this corner, regardless of which path is taken.

Let $(p_{ij})$ be the Pl\"{u}cker coordinates of $V^s(0)$. From (\ref{Plucker xing condition}), we see that $z^*$ is a conjugate time if and only if  \begin{equation}\label{corner1 conjugate point finder}
0=-e^{-z^*\sqrt{2}/2}p_{24}+kp_{23}+p_{14}-ke^{z^*\sqrt{2}/2}p_{13}.
\end{equation}
To prove that the expression in (\ref{corner1 conjugate point finder}) never vanishes, we first calculate using (\ref{stable subspace corner1 basis}) and (\ref{Plucker coords}) that \begin{equation}
\begin{aligned}
& -p_{24}  = a(1-2a)(3-2a)(1-a)>0\\
& kp_{23}=p_{14} =-2a(1-2a)(3-2a)<0\\
& -kp_{13} = 16a(1-a)>0.
\end{aligned}
\end{equation} As a function of $z$, the right-hand side of (\ref{corner1 conjugate point finder}) can therefore be written as \begin{equation}
h(z):=Ae^{-z\sqrt{2}/2}-B+Ce^{z\sqrt{2}/2},
\end{equation} with $A,B,C>0$. It is clear that $h(z)>0$ for $|z|$ sufficiently large. Furthermore, $h$ has a single local minimum at $z=\ln(A/C)/\sqrt{2}$, at which point $h(z)=2\sqrt{AC}-B$. To show that there are no conjugate points for $\gamma_+$, it therefore suffices to show that $2\sqrt{AC}-B>0$. We compute \begin{equation}
\begin{aligned}
2\sqrt{AC}-B & = 8a(1-a)\sqrt{(1-2a)(3-2a)}-4a(1-2a)(3-2a)\\
& = 4a\left(2(1-a)\sqrt{(1-2a)(3-2a)}-(1-2a)(3-2a)\right)\\
& = 4a\sqrt{(1-2a)(3-2a)}\left(\sqrt{4(1-a)^2}-\sqrt{(1-2a)(3-2a)}\right)\\
&=4a\sqrt{(1-2a)(3-2a)}\left(\frac{1}{\sqrt{4(1-a)^2}+\sqrt{(1-2a)(3-2a)}}\right)>0,
\end{aligned}
\end{equation} as desired. This proves that the connecting orbit $\gamma_+$ from $X_{13}$ to $X_{24}$ has no conjugate points, and by the argument above the same is true of $\gamma_-$. We thus see that there is no contribution to the Maslov index in the corner near $p$.

\subsection{Passage Near $M_\eps^R$}
We now consider the tangent space to $W^u(0)$ as it moves by $M_\eps^R$. Since $M_0^R$ is one-dimensional, it will be helpful to think of the curve $T_\varphi W^u(0)$ as being parametrized by $v$ (and sometimes $u$). As noted in previous sections, $\varphi'(z)\in E^u(0,z)$ is tangent to leading order to $T_\varphi(z) M^R_\eps$ for this part of the journey. Due to its being crushed against $W^u(M_\eps^R)$, the other vector spanning $T_\varphi W^u(0)$ will be in an unstable direction, which must be $\eta_4=\eta_4(v)$ by the symplectic considerations. As was the case for the fast jumps, we are free to take $\eps=0$ due to the robustness of transverse crossings. This time, the limit $\eps\rightarrow 0$ is the singular limit on the slow timescale. At any point $P=(u,v,0,y)$ on $M_0^R$, the shooting manifold therefore has tangent space \begin{equation}
T_P W^u(0)=\mathrm{sp}\left\{\left[\begin{array}{c}
1\\
f'(u)\\
0\\
\frac{1}{c}(\gamma f'(u)-1)
\end{array}\right],\left[\begin{array}{c}
f'(u)\\
0\\
f'(u)\mu_4(u)\\
\mu_4(u)
\end{array} \right] \right\}.
\end{equation}
In this section, we must be more careful about the reference plane. The cubic is symmetric about its inflection point, meaning that \begin{equation}\label{cubc symmetry}
f'(1/3(a+1)+u)=f'(1/3(a+1)-u).
\end{equation} In particular, this implies that $f'(0)=f'(u^*)$, and hence the linearization of (\ref{FHN layer problem}) at the two jump-off points $0$ and $q$ has the same set of eigenvectors and eigenvalues. This is problematic, since $\vp'(z)$ approaches $q$ in the direction $\eta_2(q)$, which we now see is in the subspace $V^s(0)$. Moreover, we cannot use any perturbation arguments at this point, since there is another non-smooth (in the limit) reorientation at $q$ to prepare for the jump back to $M_0^L$. To sidestep this issue, we simply use the reference plane (\ref{ref plane}), with $\tau$ chosen so that $\vp(\tau)$ is on the slow manifold $M_0^L$, but not at $0$ or the landing point $\hat{q}$. We will see that this slides the conjugate point down $M_\eps^R$ to a point safely away from either corner. Now, it is clear that a crossing occurs at a point $(u,v,w,y)$ if and only if \begin{equation}\label{M_R conjugate point}
\det\left[\begin{array}{c c c c}
f'(u_\tau) & 1 & 1 & f'(u)\\
0 & f'(u_\tau) & f'(u) & 0\\
f'(u_\tau)\mu_1(u_\tau) & 0 & 0 & f'(u)\mu_4(u)\\
\mu_1(u_\tau) & \frac{1}{c}(\gamma f'(u_\tau)-1) & \frac{1}{c}(\gamma f'(u)-1) & \mu_4(u)
\end{array}\right]= 0.
\end{equation} For sure, the expression in (\ref{M_R conjugate point}) vanishes at least once. Indeed, $u$ ranges from $1$ to $u^*=2/3(a+1)$ on $M_0^R$, so since $2/3(a-1/2)<u_\tau<0$, it follows from (\ref{cubc symmetry}) that $u$ must attain the unique value $u_*$ such that $f'(u_*)=f'(u_\tau)$. At this point (call it $\vp(z_*))$, $\vp'(z_*)= T_{\vp(z_*)}M^R_0$ is parallel to $\eta_2(u_\tau)\in E^u(0,\tau)$, which means that $z_*$ is a conjugate point. At any other point on $M_0^R$, a tedious (but routine) calculation of the determinant in (\ref{M_R conjugate point}) reveals that it does not vanish, hence there are no other other conjugate points on this segment.

To calculate the contribution to the Maslov index, we need the dimension and direction of the single crossing, which occurs at the point $P_*:=\vp(z_*)=(u_*,v_*,w_*,y_*)$ and time $z=z_*$. Since $f'(u_\tau)=f'(u_*)$ but $\mu_1(u_\tau)\neq\mu_4(u_*)$, it is clear from (\ref{M_R conjugate point}) that the intersection $E^s(0,\tau)\cap T_{P_*}W^u(0)$ is one-dimensional, spanned by $\eta_2(u_\tau)=\eta_2(u_*)$, the velocity of $\varphi$. For the direction of the crossing, observe that (\ref{xing form defn}) evaluated on the velocity $\varphi'$ at a conjugate time $z^*$ can be rewritten \begin{equation}
\omega(\varphi',\frac{d}{dz}\varphi')|_{z=z_*}=\eps\omega(\varphi',\frac{d}{d\zeta}\varphi')|_{\zeta=\eps z_*}.
\end{equation} Since we only care about the sign of this expression, we can ignore the $\eps$ in front. Furthermore, $\varphi'\approx\eta_2(v)$ along $M_\eps^R$, and $v$ increases as $\zeta$ increases for the reduced flow, so it follows that \begin{equation}
\mathrm{sign }\,\Gamma(E^u,E^s(0,\tau);z_*)(\varphi'(z_*))=\mathrm{sign }\,\omega(\eta_2(v),\partial_v\eta_2(v))|_{v=v_*}=\frac{g''(v_*)}{c}>0.
\end{equation} In the above calculation, $g=f^{-1}$, so $g''(v_*)=-f''(u_*)/(f'(u_*))^3<0$. This shows that the crossing near the slow manifold contributes $+1$ to the Maslov index, so it offsets the crossing in the opposite direction along the fast jump.

\subsection{Second Corner}
As the slow flow carries $\varphi$ up $M_\eps^R$, it approaches the jump off point $q$, which is the scene of another abrupt reorientation of $W^u(0)$. At the bottom right corner, we saw that there was no contribution to the Maslov index, irrespective of which of the two possible paths $T_{\varphi(z)}W^u(0)$ took to get to its starting position for the slow flow. Unfortunately, we will not be so lucky at the right jump-off point.

First, let us determine the correct ``boundary conditions'' for the corner problem. From the previous section, we know that $T_{\varphi(z)}W^u(0)$ will be $O(\eps)$ close to $X_{24}$ as $\vp(z)$ approaches $q$. In this subsection, the notation $X_{ij}$ refers to the plane $\mathrm{sp}\{\eta_i(q),\eta_j(q)\}$ spanned by eigenvectors of the linearization (\ref{linearization around wave}) evaluated at $q$, for which $u=u^*$ and $f'(u^*)=-a$. The exit position of $\varphi$ along the back can be determined by using the singular solution; as was the case for the front, the wave will be launched from $q$ in the weak unstable direction, $\eta_3(u^*)$. We argue that the second direction present is the most unstable eigenvector $\eta_4(u^*)$. Indeed, the tangent vectors to $W^u(0)$ solve (\ref{linearization around wave}), which is essentially autonomous in the neighborhood of $q$. We know that the initial condition will be $O(\eps)$ close to $\eta_4$, and therefore this direction must dominate near the corner, since it is the direction of most rapid growth for the autonomous system. It follows that we are searching for a heteroclinic connection from $X_{24}$ to $X_{34}$.

From Appendix B, we know that $W^u(X_{24})$ is one-dimensional and $X_{34}$ is a global attractor, so the only way to move from one point to the other in $\Lambda(2)$ is to exchange $\eta_2$ for $\eta_3$ in the basis for $T_{\varphi(z)}W^u(0)$. There are again two orbits that make this connection, $\gamma_+$ through $\mathrm{sp}\{\eta_2+\eta_3,\eta_4\}$ and $\gamma_-$ through $\mathrm{sp}\{\eta_2-\eta_3,\eta_4\}$. As before, it is easy to find the solutions for the equation induced on $\Lambda(2)$ by the constant coefficient system \begin{equation}\label{corner2 linear problem}
\left(\begin{array}{c}
p\\
q\\
r\\
s
\end{array}\right)'=\left(\begin{array}{c c c c }
0 & 0 & 1 & 0\\
0 & 0 & 0 & 0\\
a & 1 & -c & 0\\
-1 & \gamma & 0 & -c
\end{array}\right)\left(\begin{array}{c}
p\\q\\r\\s
\end{array}\right)
\end{equation} pinned at the corner $q$. In Pl\"{u}cker coordinates, these two paths are \begin{equation}\label{corner2 links}
\gamma_\pm(z)=(0,0,0,0,1,\pm e^{(\mu_3-\mu_2)z}).
\end{equation}
This time the concatenated path $\gamma_0=\gamma_+*-\gamma_-$ has Maslov index $1$. Indeed, one can re-parametrize $\gamma_0$ to see that it has the same homotopy class as \begin{equation}
\tilde{\gamma}_0=\begin{cases}
(0,0,0,0,1-t,t) & t\in[0,1]\\
(0,0,0,0,t-1,2-t) & t\in[1,2].
\end{cases}
\end{equation} (This is a loop since the coordinates are homogeneous.) We are now free to use any reference plane to compute the Maslov index, so we choose the convenient subspace $V=\mathrm{sp}\{\eta_1,\eta_3 \}$. The train of $V$ is given by $p_{24}=0$, so one sees that there is a unique conjugate point for $\tilde{\gamma}_0$ at $t=1$, with $\eta_3$ spanning the intersection. It is not difficult to see that this crossing is regular, so it contributes $\pm 1$ to the Maslov index. (The sign is not important.) This is the only crossing, so by homotopy invariance of the Maslov index (Theorem 2.3 (Homotopy) of \cite{RS93}), it follows that the Maslov index of ${\gamma}_0$ is $\pm 1$. Thus the Maslov indicies of $\gamma_+$ and $\gamma_-$--which are integers summing to $\pm 1$--must be different.

\begin{figure}[h]
	\centering
	\begin{subfigure}{0.49\linewidth} \centering
		\includegraphics[scale=0.75]{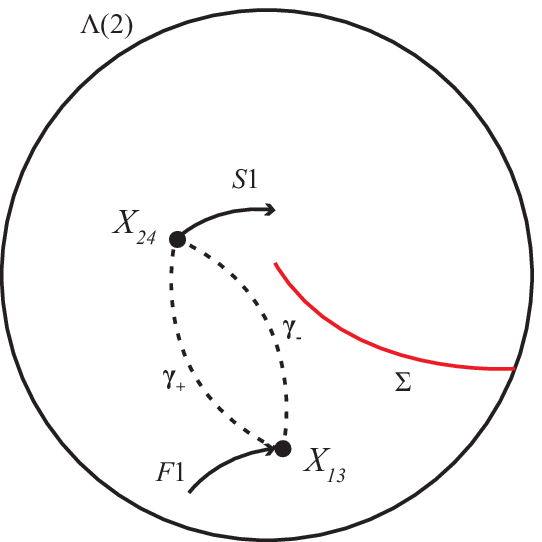}
		\caption{Corner near $p$}
	\end{subfigure}
	\begin{subfigure}{0.49\linewidth} \centering
		\includegraphics[scale=0.75]{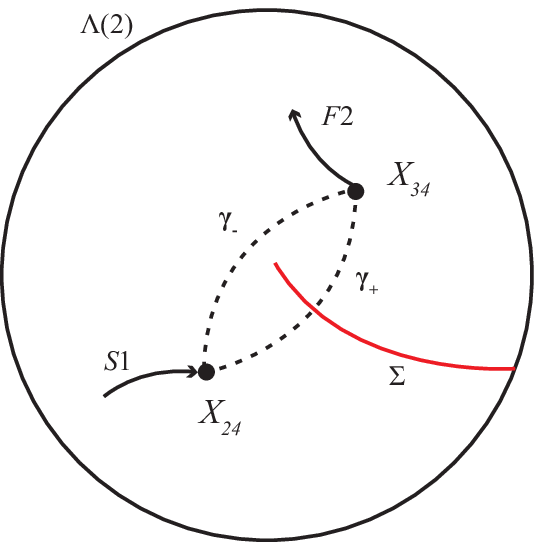}
		\caption{Corner near $q$}
	\end{subfigure}
	\caption{Schematic of corners where transitions occur between fast and slow dynamics.}
\end{figure}

We actually make the stronger claim that one of the indices is $0$ and the other is $\pm 1$. To see this, recall from (\ref{detection form}) that we detect crossings by evaluating a fixed one-form on (\ref{corner2 links}). Doing so yields a monotone function in $z$, which can have only $0$ or $1$ zeros. It therefore suffices to check the sign of this one-form at the endpoints (i.e. $|z|\gg 1$) of the correct curve. Before doing so, we must determine which of $\gamma_+$ and $\gamma_-$ is traversed to connect the two states.

The scalings of $\eta_2$ and $\eta_3$ are important for distinguishing the paths $\gamma_+(z)$ and $\gamma_-(z)$, so we fix the basis vectors \begin{equation}\label{corner2 evecs}
\eta_2=\left[\begin{array}{c}
-1/a\\1\\0\\\frac{1}{c}\left(\gamma+\frac{1}{a}\right)
\end{array} \right],\hspace{.05 in}\eta_3=\left[\begin{array}{c}
0\\0\\0\\1
\end{array}\right],\hspace{.05 in}\eta_4=\left[\begin{array}{c}
-a\\0\\-a\frac{\sqrt{2}}{2}\\\frac{\sqrt{2}}{2}
\end{array}\right].
\end{equation}
As explained above, one tangent direction to the shooting manifold is $\eta_4$, which will not move in the limit, since it is an eigenvector of (\ref{corner2 linear problem}). It is therefore evident that the trajectory in $\Lambda(2)$ is driven by the change in the velocity $\varphi'$. To see which of the paths $\gamma_\pm$ is taken, we must know the sign of the multiple of $\eta_2$ (resp. $\eta_3$) that $\varphi'$ is upon entrance to (resp. exit from) a neighborhood of $q$. The entrance is clear from the slow flow (\ref{slow flow}); to leading order, $v$ is increasing, $u$ is decreasing, $w\approx 0$ and $y$ is decreasing, hence $\varphi'$ is a positive multiple of $\eta_2$, comparing with (\ref{corner2 evecs}).

To see the orientation at exit, set $\tilde{y}=y-\frac{1}{c}(\gamma v^*-u^*)$. Along the back, $y$ goes from $\frac{1}{c}(\gamma v^*-u^*)$ at $q\in M_0^R$ to $\frac{1}{c}(\gamma v^*-(u^*-1))$ at $\hat{q}\in M_0^L$, so $\bar{y}$ goes from $0$ to $-1/c$. Furthermore, we compute that $\bar{y}$ satisfies \begin{equation}
\tilde{y}'=-y'=cy-\gamma v^*+u=-c\bar{y}-(u^*-u)=-c\bar{y}-u_f,
\end{equation} where $u_f$ is the equation for $u$ on the front, as in (\ref{McKean soln}). This is the same equation and boundary conditions satisfied by $y$ along the front, so we have \begin{equation}\label{y bar soln}
\tilde{y}(z)=Ke^{-cz}+e^{-cz}\int\limits_{-\infty}^{z}e^{cs}u_f(s)\,ds,
\end{equation} where $K$ is given by (\ref{y soln constant}). Notice that $K$ is positive, so \begin{equation}
\lim\limits_{z\rightarrow-\infty}e^{cz}y'(z)=-\lim\limits_{z\rightarrow-\infty}e^{cz}\tilde{y}'(z)=cK<0,
\end{equation} using (\ref{y bar soln}). As along the front, $u$ and $w$ still decay faster than $y$ at $-\infty$, so it follows that $\varphi'$ leaves $q$ along the back in the direction $cK\eta_3$. This proves that the connecting orbit in $\Lambda(2)$ from $X_{24}$ to $X_{34}$ is $\gamma_-$.

To determine the contribution to the Maslov index, it therefore suffices to compare the signs of $\det\left[E^s(0,\tau),\eta_2,\eta_4 \right]$ and $\det\left[E^s(0,\tau),-\eta_3,\eta_4 \right].$ Representing $E^s(0,\tau)$ in the basis (\ref{ref plane}), a calculation gives that \begin{equation}\label{corner2 in beta}
\det\left[E^s(0,\tau),\eta_2,\eta_4 \right] = \frac{\delta(\delta a\sqrt{2}+2Qa^2-\delta\mu_1(u_\tau))}{-2ac},
\end{equation} 
where $\delta=-f'(u_\tau)-a>0$ and $Q=\sqrt{2}/2-\mu_1(u_\tau)>0$. The introduction of these variables simplifies the calculation because $f'(u_\tau)$ approaches $-a$ from above as $u_\tau\rightarrow0$. It is thus clear that the determinant in (\ref{corner2 in beta}) is positive. Similarly, we compute that \begin{equation}
\det\left[E^s(0,\tau),-\eta_3,\eta_4 \right]=\frac{(f'(u_\tau))^2a(\sqrt{2}-2\mu_1(u_\tau))}{2}>0.
\end{equation} Since the detection form is monotone in $z$ on $\gamma_-(z)$, the fact that it has no changes in sign implies that it has no zeros, and therefore there are no conjugate points near $q$. To recap, the cumulative Maslov index as we enter the back is 0: $-1$ from the front $+1$ near the right slow manifold.
\subsection{Second Fast Jump}
The analysis of the back is nearly identical to that of the front, so we will skip many of the details. Along the back, $W^u(0)$ is $O(\eps)$ close to $W^u(q)$, the cylinder over the Nagumo back. As a remark, the full power of the Exchange Lemma is not needed to see this--we are not carrying any extra center/slow directions in the Maslov index calculation. We are once again free to consider intersections of $T_{q_b(z)}W^u(q)$ with the train of $V^s(0)$, since $V^s(0)$ is transverse to the tangent space to the cylinder near $q$ and $\hat{q}$. Recycling the notation $u_f(z)$ from the front, we have $u_b=u^*-u_f$, hence $w_b= -w_f$. We can again solve for $w$ as a function of $u$ to obtain \begin{equation}\label{back profile}
w(u)=-\frac{\sqrt{2}}{2}(u^*-u)(1-(u^*-u)),
\end{equation} where now $u$ ranges from $u^*$ to $u^*-1$. This yields the basis \begin{equation}
\mathrm{sp}\left\{\left[\begin{array}{c}
1\\0\\\frac{\sqrt{2}}{2}-\sqrt{2}(u^*-u)\\0
\end{array}\right],\left[\begin{array}{c}
0\\0\\0\\1
\end{array}\right]\right\}
\end{equation} of $T_{q_b(z)}W^u(q)$. Comparing with (\ref{stable subspace basis}), we see that there is a unique conjugate point, which is the value $z^*$ such that $(u^*-u)=\displaystyle\frac{1}{2}+a.$ The intersection $V^s(0)\cap T_{q_b(z)}W^u(q)$ is again spanned by $\xi=\{1,0,-a\sqrt{2},\sqrt{2} \}$. Since $f'(u^*-1/2-a)=f'(1/2+a)$ by (\ref{cubc symmetry}), the crossing form calculation is identical to (\ref{F1 crossing form}). Explicitly, we have \begin{equation}\label{F2 crossing form}
\begin{aligned}
\omega(\xi,A(0,z^*)\xi) & =-f'\left(u^*-\left(\frac{1}{2}+a\right)\right)+ca\sqrt{2}-2a^2\\
 & = -f'\left(\frac{1}{2}+a\right)+ca\sqrt{2}-2a^2\\
 & = a^2-\frac{1}{4}<0.
\end{aligned}
\end{equation} Thus the Maslov index of the second fast jump is $-1$. 

\subsection{Final Corner, Passage near $M_\eps^L$, and Return to Equilibrium}
The analysis of the corner $\hat{q}$ is identical to that of $p$. First, the symmetry of $f$ ensures that the set of eigenvectors and eigenvalues for the system (\ref{linearization around wave}) evaluated at $p$ and at $\hat{q}$ are the same when $\eps=0$. Also, the tangent space of $W^u(0)$ is $O(\eps)$ close to $X_{13}$ upon entrance into a neighborhood of both points. Finally, Deng's Lemma and the already-proved existence of the wave necessitate that $T_{p_\mathrm{out}}W^u(0)$ and $T_{\hat{q}_\mathrm{out}}W^u(0)$ are both $O(\eps)$ close to $X_{24}$. Since there are only two possible paths of (Lagrangian) planes connecting $X_{13}$ and $X_{24}$--neither of which has any conjugate points--there is no need to investigate the corner $\hat{q}$ further. We therefore turn our attention to the slow return to equilibrium.

As for $M_\eps^R$, we expect one conjugate point for the final slow piece. This one is actually easier to find; by definition of $\Maslov$, there is a conjugate point at $z=\tau$, for which value of $z$ we have \begin{equation}
\mathrm{sp}\{\vp'(\tau)\}=E^u(0,\tau)\cap E^s(0,\tau).
\end{equation} The fact that $\vp$ is transversely constructed implies that the intersection is only one-dimensional. In terms of the singular orbit, we see that the intersection is spanned by the tangent vector to $T_{\vp(\tau)}M_\eps^L$. The non-existence of any other conjugate points is identical to \S 4.3--one simply shows that determinant which detects conjugate points does not vanish unless $u=u_\tau$.

The sign of this crossing is computed as in \S 5.3. This time, we have \begin{equation}
\mathrm{sign}\,\Gamma(E^u(0,\cdot),E^s(0,\tau);\tau)(\varphi'(\tau))=-\mathrm{sign}\,\omega\left(\eta_2(v),\partial_v\eta_2(v) \right)|_{v=v_\tau},
\end{equation} since $v$ decreases as $\zeta=\eps z$ increases on $M_0^L$. Once again defining $g(v)=f^{-1}(v)$--this time on the left branch of $M_0$--one computes from (\ref{omega defn}) that \begin{equation}
\omega\left(\eta_2(v),\partial_v\eta_2(v) \right)|_{v=v_\tau}=\frac{g''(v)}{c}<0,
\end{equation} where $g''(v)=-f''(u_\tau)/(f'(u_\tau))^3>0.$ Hence the crossing is positive, as it was for the conjugate point on $M_\eps^R$. Since this crossing occurs at the right endpoint of the curve $E^u(0,z)$, the contribution to the Maslov index is $+1$, by Definition \ref{Maslov defn}.
\subsection{Concluding Remarks}
Adding up the Maslov index of the constituent pieces, we see that \begin{equation}\label{Maslov scorecard}
\Maslov=-1+1-1+1=0.
\end{equation} This proves Theorem \ref{Morse = Maslov thm}, and we conclude that the fast traveling pulses for (\ref{general PDE}) are nonlinearly stable. Although the profiles and speeds of the waves in (\ref{general PDE}) and those in the same equation without diffusion on $v$ are very similar, we point out that the stability proofs are entirely different and independent of each other. In \cite{Jones84,Yan85}, the stability result is obtained by showing that the eigenvalues of the linearized operator are close to those for the reduced systems corresponding to the fast front and back. Conversely, the eigenvalue problem for $L$ in (\ref{linearized operator}) is analyzed entirely as an operator on $BU(\bbR,\bbR^2)$. Thus the smallness of $\eps$ in each setting appears in different ways. In \cite{corn}, it used to achieve monotonicity for the Maslov index in the spectral parameter, as well to prove that the unstable spectrum of $L$ must be real. Most notably, the small parameter allows us to calculate $\Maslov$ using geometric singular perturbation theory.  

\appendix
\section{Pl\"{u}cker Coordinates and the Detection Form}

The Maslov index is defined for the unstable bundle $E^u(0,z)$, so it is important to know how this solution space evolves. It is standard that (\ref{linearization around wave}) induces a flow on $\Gr_2(\bbR^4)$, and the easiest way to analyze this is equation is via the Pl\"{u}cker coordinates. For more background on the results contained in this section, the reader is referred to \S 3 of \cite{CJ17} and also \cite{CDB09}. Let $\{e_i\}_{i=1}^4$ be any basis of $\bbR^4$. This induces a basis $\{e_i\wedge e_j\}$ of $\bigwedge^2(\bbR^4)$. Any linear system \begin{equation}\label{gen linear system}
Y'(z)=B(z)Y(z)
\end{equation} induces an equation on $\bigwedge^2(\bbR^4)$ by the formula \begin{equation}\label{induced eqn}
\frac{d}{dz}(v_1(z)\wedge v_2(z))=B(z)v_1(z)\wedge v_2(z)+v_1(z)\wedge B(z)v_2(z).
\end{equation} To relate (\ref{induced eqn}) to the dynamics of (\ref{gen linear system}), one can use the Pl\"{u}cker embedding to realize two-dimensional subspaces of $\bbR^4$ as elements of $\bigwedge^2(\bbR^4)$. More precisely, the map \begin{equation}
\begin{aligned}
j:\Gr_2(\bbR^4) & \rightarrow \mathbb{P}({\smash\bigwedge^2} \bbR^4)\\
V=\mathrm{sp}\{u,v\} & \mapsto[u\wedge v]
\end{aligned}
\end{equation} is a well-defined embedding. (See \cite{Hassett}.) Using the definition of the wedge product, one sees that for $u=\sum u_ie_i$ and $v=\sum v_ie_i$, we have coordinates \begin{equation}\label{Plucker coords}
p_{ij}=\left|\begin{array}{c c}
u_i & v_i\\
u_j & v_j
\end{array}\right|
\end{equation} for the plane $V=\mathrm{sp}\{u,v\}$. These are called the Pl\"{u}cker coordinates of $V$, and they are homogeneous (i.e. projective) because choosing a different basis of $V$ would change the $p_{ij}$ by a constant, nonzero multiple. Using (\ref{induced eqn}) (which amounts to the product rule on the $p_{ij}$) one can write down a differential equation for the Pl\"{u}cker coordinates. For example, this is done for (\ref{linearization around wave}) in \cite{CJ17}.

The Pl\"{u}cker coordinates are useful for finding conjugate points for a curve of subspaces, which are intersections between the curve and a fixed subspace. Let $V=\mathrm{sp}\{v_1,v_2\},W=\mathrm{sp}\{w_1,w_2\}\in\Lambda(2)$ be Lagrangian planes with Pl\"{u}cker coordinates $(p_{ij})$ and $(q_{ij})$ respectively. Then \begin{equation}\label{Plucker xing condition}
\begin{aligned}
W\cap V\neq\{0\}&\iff\det[v_1,v_2,w_1,w_2] = 0\\
&=p_{12}q_{34}-p_{13}q_{24}+p_{14}q_{23}+p_{23}q_{14}-p_{24}q_{13}+p_{34}q_{12},
\end{aligned}
\end{equation} using (\ref{Plucker coords}) and cofactor expansion to the compute the determinant. Thus if one wishes to find \emph{all} conjugate points for a curve $W(z)$ of Lagrangian subspaces with respect to a reference plane $V$, then the function \begin{equation}\label{detection form}
\beta(z)=\det[V,W(z)]
\end{equation} is a linear function of the Pl\"{u}cker coordinates of $W$ whose zeros correspond to conjugate points. In \cite{CJ17}, this function is called the \emph{detection form}.

Finally, it is sometimes convenient to change the basis of $\bbR^4$ before computing the Pl\"{u}cker coordinates. For example, suppose \begin{equation}\label{constant coeff system}
Y'(z)=BY(z)
\end{equation} is a constant coefficient, linear system on $\bbR^4$. If $(\mu_i,\eta_i)$ are eigenvalue, eigenvector pairs for $B$ (assume that all eigenvalues have full geometric multiplicity), then the solution to (\ref{constant coeff system}) through $\eta_i$ is given by $e^{\mu_iz}\eta_i$. It then follows from (\ref{induced eqn}) that the solution to the equation induced by (\ref{constant coeff system}) on $\Gr_2(\bbR^4)$ through $\mathrm{sp}\{\eta_i,\eta_j\}$ is given by \begin{equation}
e^{Bz}\eta_i\wedge\eta_j+\eta_i\wedge e^{Bz}\eta_j=e^{\mu_iz}\eta_i\wedge\eta_j+\eta_i\wedge(e^{\mu_jz}\eta_j)=e^{(\mu_i+\mu_j)z}\eta_i\wedge\eta_j,
\end{equation} using the linearity of $\wedge$. In Pl\"{u}cker coordinates, this is given by \begin{equation}
p_{kl}=\begin{cases}
e^{(\mu_i+\mu_j)z} & (k,l)=(i,j)\\
0 & \text{else}
\end{cases}.
\end{equation} Since these coordinates are projective, this means that the solution is constant. This makes sense, because the eigenspaces of $B$ are invariant under (\ref{constant coeff system}).
\section{Phase Portrait of Induced Flow on $\Lambda(2)$}
The connection in $\Lambda(2)$ between the fast and slow dynamics is determined in the $\eps=0$ limit by the constant coefficient system obtained by linearizing about the relevant corner point. The phase portrait of such systems is described completely in \cite{Shayman86} and is of interest in control theory. Here we catalog the relevant results for this work, tailored to the linearization of the traveling wave ODE (\ref{FHN traveling wave ODE}) at any point on $M_0^{R/L}$. The reader should be aware that the presentation in \cite{Shayman86} assumes that the flow on $\Lambda(n)$ is given by the action of a $2n\times2n$ symplectic matrix on Lagrangian subspaces. This is not the case here, since the solution operator for (\ref{linearization around wave}) is not symplectic. However, this does not change the geometry of the flow on $\Lambda(2)$, which is an invariant manifold of the system on $\mathrm{Gr}_2(\bbR^4)$. It is therefore clear that the following facts remain true, although the assumptions of the corresponding theorems in \cite{Shayman86} sometimes require modification.

To fix some notation, first recall that there are three ``corners" at which transitions from fast-to-slow dynamics (or vice-versa) occur: $p=(1,0,0,-1/c^*)$, $q=(u^*,v^*,0,(1/c^*)(\gamma v^*-u^*))$, and $\hat{q}=(u^*-1,v^*,0,(1/c^*)(\gamma v^*-u^*+1)).$ From (\ref{evals of lin}), we know that there are four eigenvalues of the linearization at each point, which satisfy (when $\eps=0$) $\mu_1<\mu_2=0<\mu_3=-c<\mu_4$. $\mu_1$ and $\mu_4$ depend on $u$, but their sum is always equal to $-c$. Now consider (\ref{linearization around wave}), except for fixed $u$. This system is then of the form (\ref{constant coeff system}), and it induces a flow on $\mathrm{Gr}_2(\bbR^4)$. Explicitly, the trajectory through any plane $V\in\mathrm{Gr}_2(\bbR^4)$ is given by $\exp(Bz)\cdot V$. Thus a subspace is an equilibrium for (\ref{constant coeff system}) if and only if it is $B$-invariant. There are six such fixed points, given by \begin{equation}
X_{ij}=\mathrm{sp}\{\eta_i,\eta_j\}, \hspace{.1 in}\{i,j\}\in{\{1,2,3,4\} \choose 2}.
\end{equation} Of these, $X_{12},X_{13},X_{24}$ and $X_{34}$ are Lagrangian planes, making them the points of interest. The following theorem holds for the flow on $\Lambda(2)$ induced by the constant coefficient system (\ref{constant coeff system}) based at each corner point mentioned above. Recall that $\dim\Lambda(2)=3$. We refer the reader to \cite{Shayman86} for proofs.

\begin{theorem}[Shayman \cite{Shayman86}]
For the equation induced by (\ref{constant coeff system}) on $\Lambda(2)$, the following are true:
\begin{enumerate}
\item Each fixed point for (\ref{constant coeff system}) is hyperbolic. We have \begin{equation}
\begin{aligned}
\dim W^u(X_{12}) &=\dim W^s(X_{34})=3\\ \dim W^u(X_{13}) &=\dim W^s(X_{24})=2\\ \dim W^u(X_{24}) &=\dim W^s(X_{13})=1.
\end{aligned}
\end{equation} Furthermore, each of $W^u(X_{12})$ and $W^s(X_{34})$ is open and dense in $\Lambda(2)$.
\vspace{.1 in}
\item $\Lambda(n)=\bigcup W^u(X_{ij})=\bigcup W^s(X_{ij})$, due to the fact that $\Lambda(2)$ is compact.
\vspace{.1 in}
\item Each $W^{u/s}(X_{ij})$ is a Schubert cell. In particular, it is diffeomorphic to $\bbR^d$, where $d$ is the dimension of the invariant manifold.
\item For any $i,j,i',j'$, either $W^u(X_{ij})\cap W^s(X_{i'j'})=\phi$ or $W^u(X_{ij})\pitchfork W^s(X_{i'j'})$.
\end{enumerate}
\end{theorem}

\begin{rem}
It is also true that the vector field on $\Lambda(2)$ induced by (\ref{constant coeff system}) is Morse-Smale.
\end{rem}

\bibliographystyle{amsplain}
\bibliography{FHN_Maslov2}

\end{document}